\documentclass{amsart}

\usepackage{geometry, amssymb, verbatim, esint}
\usepackage{color}
\usepackage{graphicx}

\usepackage[colorlinks=true,urlcolor=blue, citecolor=red,linkcolor=blue,
linktocpage,pdfpagelabels, bookmarksnumbered,bookmarksopen]{hyperref}
\usepackage[hyperpageref]{backref}
\usepackage{mathrsfs}
\newtheorem{teo}{Theorem}[section]
\newtheorem{lm}[teo]{Lemma}
\newtheorem{coro}[teo]{Corollary}
\newtheorem{prop}[teo]{Proposition}
\newtheorem*{main}{Main Theorem}
\theoremstyle{definition}

\newtheorem{definition}[teo]{Definition}
\newtheorem{exa}[teo]{Example}
\newtheorem{rem}[teo]{Remark}
\newtheorem*{ack}{Acknowledgments}

\numberwithin{equation}{section}

\title[Periodic sets]{Extremals for sharp\\Poincar\'e-Sobolev inequalities:\\ periodically perforated sets and beyond}
\date{\today}

\subjclass[2010]{35A15, 35P30, 46E35}
\keywords{Poincar\'e-Sobolev inequalities, Lane-Emden equation, lack of compactness, capacity.}

\author[Brasco]{Lorenzo Brasco}
\address[L.\ Brasco]{Dipartimento di Matematica e Informatica
	\newline\indent
	Universit\`a degli Studi di Ferrara
	\newline\indent
	Via Machiavelli 35, 44121 Ferrara, Italy}
\email{lorenzo.brasco@unife.it}

\author[Briani]{Luca Briani}
\address[L.\ Briani]{School of Computation, Information and Technology,		\newline\indent
 	Technical University of Munich
 	\newline\indent
	Boltzmannstra\ss e 3, 85748 Garching bei M\"unchen, Germany.}
\email{luca.briani@tum.de}

\author[Prinari]{Francesca Prinari}

\address[F. Prinari]{Dipartimento di Scienze Agrarie, Alimentari e Agro-ambientali
\newline\indent 
Universit\`a di Pisa
\newline\indent
Via del Borghetto 80, 56124 Pisa, Italy}
\email{francesca.prinari@unipi.it}

\begin{document}

\begin{abstract}
We consider periodically perforated unbounded open sets and prove existence of extremals for the relevant sharp Poincar\'e-Sobolev embedding constant. The existence result holds no matter the shape or the regularity of the hole: it is sufficient that the latter is a compact set with positive capacity. We also show how to apply the main result in order to get a similar existence statement, for sets which are periodic in some directions and bounded in all the others.
\end{abstract}

\maketitle

\begin{center}
\begin{minipage}{10cm}
\small
\tableofcontents
\end{minipage}
\end{center}

\section{Introduction}

\subsection{A bestiary of open sets}
In this paper, we wish to pursue a classification of families of open sets $\Omega\subseteq\mathbb{R}^N$ for which the embedding 
\begin{equation}
\label{mario}
\mathscr{D}^{1,p}_0(\Omega)\hookrightarrow L^q(\Omega),
\end{equation}
has the following properties: 
\begin{itemize}
\item[(E1)] it is continuous;
\vskip.2cm 
\item[(E2)] it may fail to be compact;
\vskip.2cm
\item[(E3)] nevertheless, the sharp embedding constant
\[
\lambda_{p,q}(\Omega):=\inf_{\varphi\in C^\infty_0(\Omega)}\left\{\int_\Omega |\nabla\varphi|^p\,dx\, :\, \|\varphi\|_{L^q(\Omega)}=1\right\},
\]
is attained in $\mathscr{D}^{1,p}_0(\Omega)$, the latter being the {\it homogeneous Sobolev space} obtained as the completion of $C^\infty_0(\Omega)$ with respect to the norm
\[
\varphi\mapsto \|\nabla \varphi\|_{L^p(\Omega)},\qquad \text{for every}\ \varphi\in C^\infty_0(\Omega).
\] 
\end{itemize}
The exponent $q$ of the target space in \eqref{mario} will be {\it subcritical} in the sense of Sobolev embeddings, that is 
\begin{equation}\label{exponents} 
q \left\{\begin{array}{ll}
< p^*,& \text{if}\ p<N,\\
<\infty, & \text{if}\ p=N,\\
\le \infty,& \text{if}\ p> N,
\end{array}
\right.\qquad \text{where}\ p^*=\frac{N\,p}{N-p}.
\end{equation}
Moreover, we will limit ourselves to the {\it super-homogeneous} case $q>p$. The reason for this last restriction is readily explained: in the case $q<p$, a (maybe not so) well-known result by Maz'ya asserts that \eqref{mario} is continuous if and only if it is compact (see \cite[Theorem 15.6.2]{Maz}). Thus, it is not possible to find open sets having the aforementioned interesting features.
\par
The borderline case $q=p$ would be interesting, but it will not be treated here, for a reason which we will explain in a while. For the moment, we just fix the distinguished notation
\[
\lambda_{p}(\Omega):=\inf_{\varphi\in C^\infty_0(\Omega)}\left\{\int_\Omega |\nabla\varphi|^p\,dx\, :\, \|\varphi\|_{L^p(\Omega)}=1\right\},
\]
for this case.
\vskip.2cm\noindent
Before explaining in details our scopes, a couple of disclaimers are in order. Firstly, in this paper we will use the term {\it extremals} to denote the minimizers for the quantity $\lambda_{p,q}(\Omega)$ defined above or, more generally, any function $u\in \mathscr{D}^{1,p}_0(\Omega)\setminus\{0\}$ such that
\[
\int_\Omega |\nabla u|^p\,dx=\lambda_{p,q}(\Omega)\,\left(\int_\Omega |u|^q\,dx\right)^\frac{p}{q}.
\]
Secondly, concerning the space $\mathscr{D}^{1,p}_0(\Omega)$, we recall that 
\begin{equation}
\label{mara}
\lambda_p(\Omega)>0\qquad \Longleftrightarrow\qquad \lambda_{p,q}(\Omega)>0,
\end{equation}
for $q$ as above (see for example \cite[Theorem 15.4.1]{Maz}).
Thus, for the sets we are interested in, the following two norms on $C^\infty_0(\Omega)$
\[
\|\nabla \varphi\|_{L^p(\Omega)}\qquad \text{and}\qquad \|\varphi\|_{W^{1,p}(\Omega)}=\|\varphi\|_{L^p(\Omega)}+\|\nabla \varphi\|_{L^p(\Omega)}, 
\]
are equivalent. Accordingly, the homogeneous space $\mathscr{D}^{1,p}_0(\Omega)$ coincides with the more familiar space $W^{1,p}_0(\Omega)$, i.e. the closure of $C^\infty_0(\Omega)$ in the standard (non-homogeneous) Sobolev space $W^{1,p}(\Omega)$. For this reason, from now on we will work directly with the space $W^{1,p}_0(\Omega)$.
\vskip.2cm\noindent
We can now get closer to the scopes of the manuscript.
In our recent paper \cite{BraBriPri}, we have shown that {\it Steiner symmetric open sets} (not coinciding with the whole $\mathbb{R}^N$) fall within the class of open sets having properties (E1)--(E3). Thus, for instance, we have existence of extremals for a {\it slab}, i.e. an open set of the form $\Omega=\mathbb{R}^{N-1}\times(-1,1)$ (for $p=N=2$ this result was originally contained in \cite{AmTo, BBT}). For this set the embedding \eqref{mario} clearly fails to be compact, because of the translation invariance along the first $N-1$ coordinate directions.
\par
In this paper, we wish to consider {\it periodically perforated open sets}, i.e. open sets obtained by removing from $\mathbb{R}^N$ a periodic array of translated copies of a fixed ``hole'' $K$. We refer to the next subsection for the precise definition (see Definition \ref{defi:apps_intro} below), but the reader could keep in mind the basic example
\begin{equation}
\label{forellato}
\Omega=\mathbb{R}^N\setminus \left(\bigcup_{\mathbf{i}\in\mathbb{Z}^N}\overline{B_r(\mathbf{i})}\right),\qquad \text{for some}\ 0< r<\frac{1}{2},
\end{equation}
just to fix ideas.
We will see in a moment that we can treat much more general situations, where spherical holes are replaced by arbitrary shapes. The only requirement will be that the ``hole'' $K$ must be non-negligible, in the sense of {\it $p-$capacity}. 
\begin{rem}
It is not difficult to see that the Poincar\'e constant $\lambda_p(\Omega)$ can not be attained for a set like \eqref{forellato}. For this reason, we do not consider the case $q=p$. Indeed, it is well-known that extremals would be ``unique'' in this case, in the sense that they form a one-dimensional vector space (see for example \cite[Theorem 1.3]{KLP}). This uniqueness property, in conjunction with the periodicity of the set, would in turn imply that extremals should be periodic, thus violating the membership to $W^{1,p}_0(\Omega)$.
\end{rem}

\subsection{Main results}
We now introduce the class of open sets we wish to consider in this paper. We fix at first some further notation. We set 
\[
Q_{1/2}=\left(-\frac{1}{2},\frac{1}{2}\right)\times\dots \times\left(-\frac{1}{2},\frac{1}{2}\right)=\left(-\frac{1}{2},\frac{1}{2}\right)^N.
\]
For every given vector $\mathbf{t}=(t_1,\dots,t_N)$ with $t_i>0$, we define the following anisotropic dilation operator
\[
\begin{array}{ccccl}
D_{\mathbf{t}}&:&\mathbb{R}^N&\to&\mathbb{R}^N\\
&& x&\mapsto& D_{\mathbf{t}}(x)=(t_1\,x_1,\dots,t_N\,x_N).
\end{array}
\]
Accordingly, for every $\mathbf{t}=(t_1,\dots,t_N)$ with $t_i>0$, we have
\[
D_{\mathbf{t}}(Q_{1/2})=\left(-\frac{t_1}{2},\frac{t_1}{2}\right)\times\dots\times \left(-\frac{t_N}{2},\frac{t_N}{2}\right).
\]
\begin{definition}
\label{defi:apps_intro}
Let $\mathbf{t}=(t_1,\dots,t_N)\in\mathbb{R}^N$ be such that $t_i>0$, for every $i\in\{1,\dots,N\}$. We also take a nonempty compact set $K\subsetneq \overline{Q_{1/2}}$. We say that $\Omega\subseteq\mathbb{R}^N$ is a {\it $\mathbf{t}-$periodically perforated open set, generated by $K$}, if it has the following form
\[
\Omega=\mathbb{R}^N\setminus\left(\bigcup_{\mathbf{i}\in\mathbb{Z}^N}D_{\mathbf{t}}(\mathbf{i}+K)\right),
\] 
see Figures \ref{fig:cubo_base} and \ref{fig:periodic_set}.
\end{definition}
\begin{figure}
\includegraphics[scale=.2]{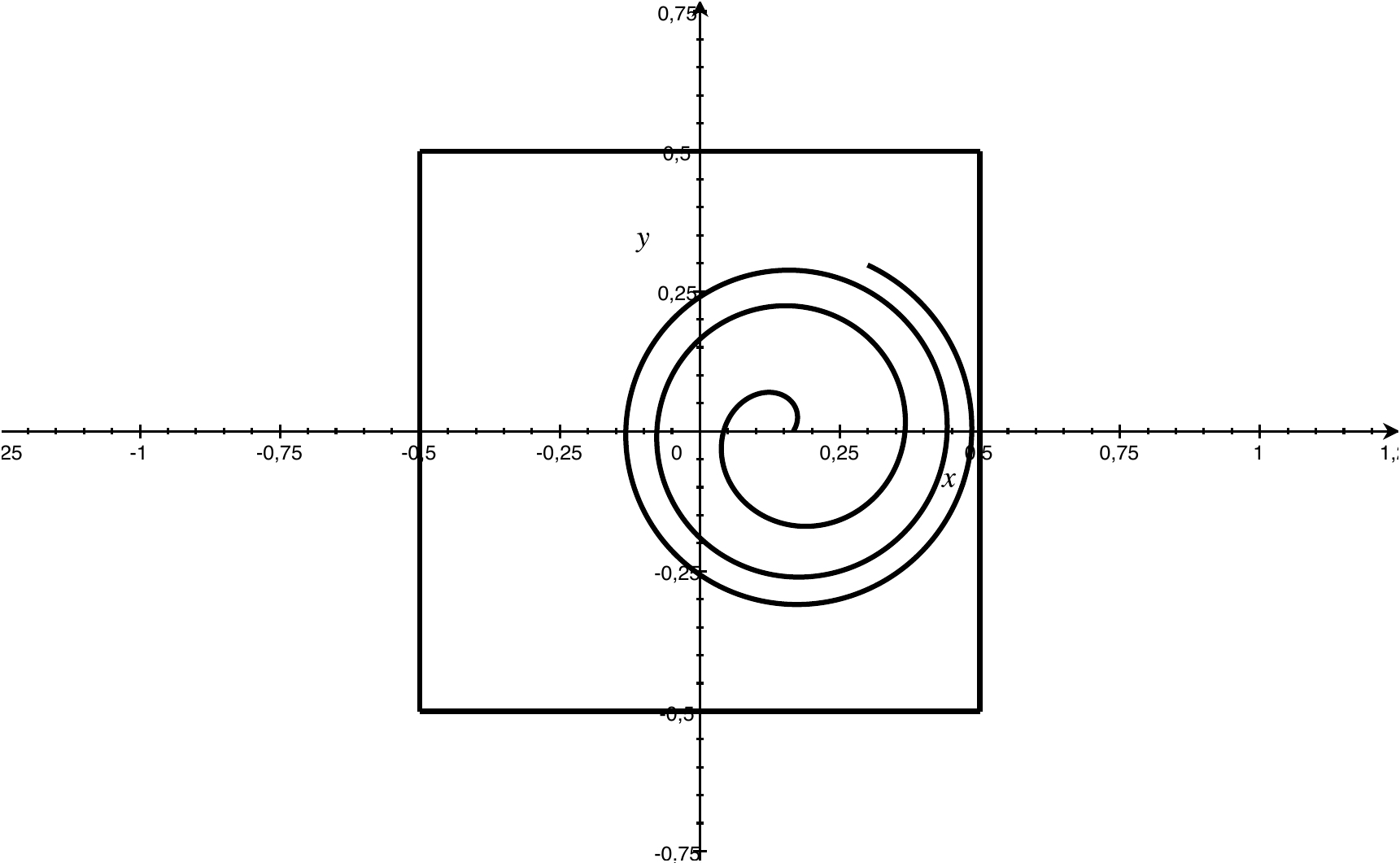}
\caption{The basic cube $Q_{1/2}$ containing a compact set $K$ (in bold line).}
\label{fig:cubo_base}
\end{figure}
\begin{figure}
\includegraphics[scale=.2]{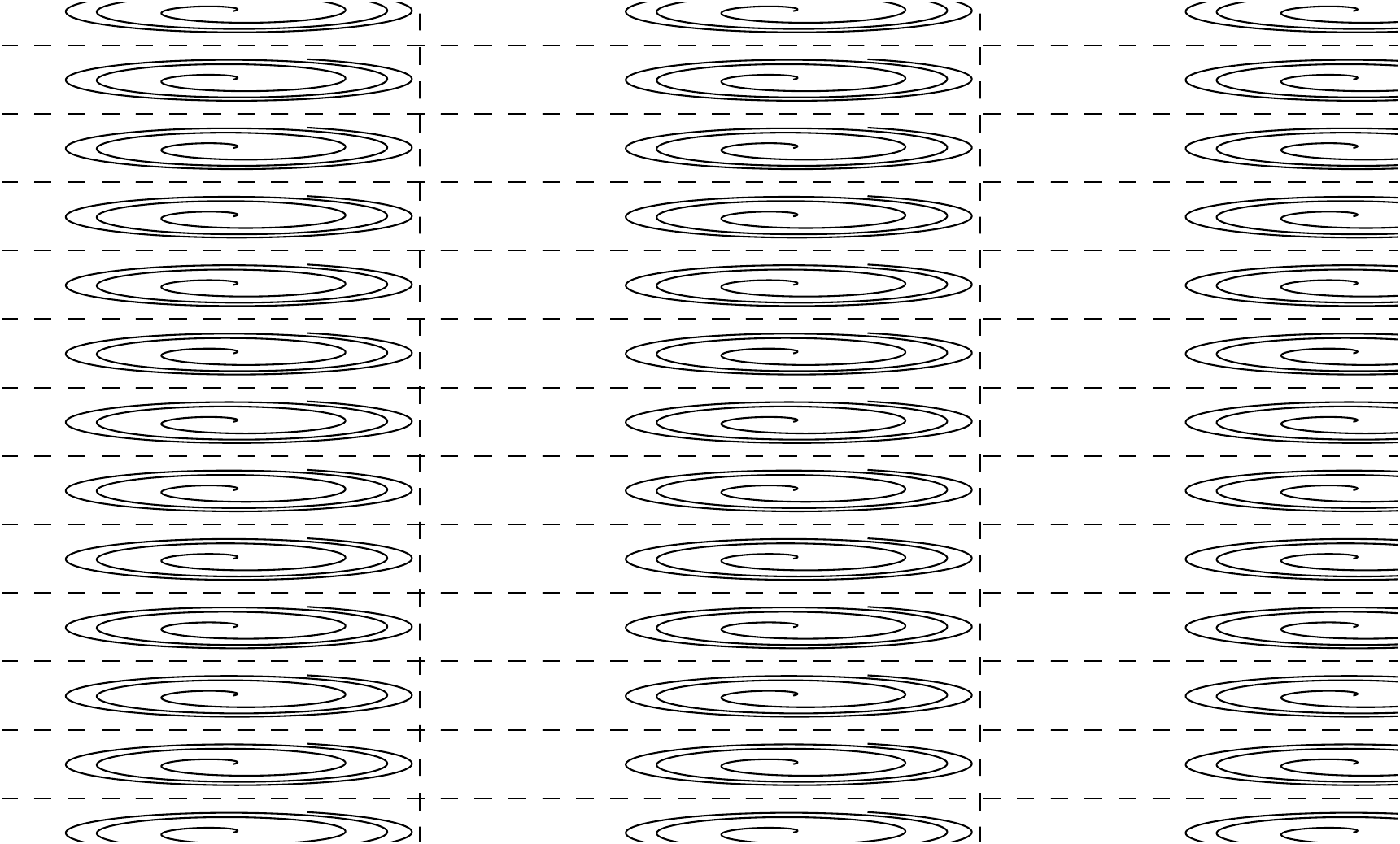}
\caption{A periodically perforated open set $\Omega\subseteq\mathbb{R}^2$, generated by $K$. Here we took $t_1=2$, $t_2=1/2$, i.e. the set is obtained by gluing together translated copies of the cube in Figure \ref{fig:cubo_base}, which has been stretched horizontally by a factor $2$ and compressed vertically by a factor $1/2$.}
\label{fig:periodic_set}
\end{figure}
\begin{rem}
Observe that, by construction, a $\mathbf{t}-$periodically perforated open set, generated by $K$, is $t_i-$periodic in direction $\mathbf{e}_i$, for every $i\in\{1,\dots,N\}$. In other words, we have
\[
\Omega=\Omega+k\,t_i\,\mathbf{e}_i,\qquad \text{for every}\ k\in\mathbb{Z},\ i\in\{1,\dots,N\}.
\]
\end{rem}
With the previous definition at hand, we can state the main existence result of this paper.
\begin{main}
Let $1<p<\infty$ and assume that  $q>p$ satisfies \eqref{exponents}. Let $\mathbf{t}=(t_1,\dots,t_N)\in\mathbb{R}^N$ be such that $t_i>0$, for every $i\in\{1,\dots,N\}$ and let $K\subsetneq \overline{Q_{1/2}}$ be a compact set such that
\[
\mathrm{cap}_p(K; Q_1)>0.
\]
Let $\Omega\subseteq\mathbb{R}^N$ be a $\mathbf{t}-$periodically perforated open set, generated by $K$.
Then, there exists a positive extremal $u\in W^{1,p}_0(\Omega)\cap L^\infty(\Omega)$.
\end{main}
\begin{rem}
The capacitary assumption on $K$ is sharp, i.e. if the condition is not satisfied then we can not have existence of extremals. Indeed, in this case we would have $\lambda_p(\Omega)=0$ (see Lemma \ref{lm:nonbanale} below) and thus, on account of \eqref{mara}, it would result $\lambda_{p,q}(\Omega)=0$, as well. We also recall that our extremals have exponential decay at infinity, on account of \cite[Theorem 7.3]{BraBriPri}. 
\end{rem}
The reader will find the proof of the previous result in Section \ref{sec:6}: for ease of readability, we found it useful to treat separately the case $q<\infty$ (Theorem \ref{teo:main}) and the extremal case $q=\infty$ (Theorem \ref{teo:maininfty}).
\par
We refer to the next subsection for some comments on the proof and a comparison with existing related results.
Here, we rather want to anticipate a possible reader's objection: at a first glance, the class of open sets considered by our Main Theorem may seem very specific. 
On the one hand, once proved this result, we will show how to apply it in order to deduce (with a minimal effort) the existence of extremals for a much wider class of open sets: these are periodic only in the first $k$ coordinate directions $\mathbf{e}_1,\dots,\mathbf{e}_k$ and bounded in the others $\mathbf{e}_{k+1},\dots,\mathbf{e}_N$. 
As a particular case, we can retrieve (with a different proof) the existence result given by Esteban in \cite[Theorems 1 and 6]{Es}, in the case of a power-type nonlinearity. We refer the reader to Theorem \ref{teo:pupazzo} and Corollary \ref{coro:Es} for the precise statements, as well as the proofs.
\par
On the other hand, we will give a couple of examples, obtained by slightly perturbing the ``model'' set \eqref{forellato}, which  show that, in general,  breaking the periodicity structure will  lead to a non-existence result. More precisely, by enlarging one single hole in \eqref{forellato}, one can easily prove that existence of extremals is lost (see Section \ref{exa:perforato_exa}). A curious phenomenon appears if, on the contrary, we shrink a little bit a single hole in \eqref{forellato}: in this case, we still have existence of extremals for the new set $\widetilde\Omega$ (see Section \ref{exa:perforato_exa2}). This is due to the following subtle fact, which deserves a comment: 
one can show that 
\[
\lambda_{p,q}(\widetilde\Omega)\lneq \lim_{R\to+\infty} \lambda_{p,q}(\widetilde\Omega\setminus \overline{B_R}).
\]
Since the global Poincar\'e-Sobolev constant is strictly less than that ``at infinity'', we get that minimizing sequences have the interest to stay ``confined'' and not to escape at infinity. This results in a gain of compactness and thus existence of extremals  holds. 
\par
This situation is however much simpler than that of our Main Theorem, where the previous argument does not apply: indeed, because of the periodicity, in general we have that  the two Poincar\'e-Sobolev constants above do coincide in the case of our sets. 
\subsection{Comments on the proof}
It is fair to declare that our main result is not a complete novelty. In order to put it in the right context, we first recall that an extremal $u$ for $\lambda_{p,q}(\Omega)$ is a weak {\it minimal energy solution} of the Lane-Emden equation 
\begin{equation}
\label{renato}
-\Delta_p v=|v|^{q-2}\,v,\ \text{in}\ \Omega,\qquad v=0,\ \text{on}\ \partial\Omega,
\end{equation}
up to a suitable ``vertical'' scaling. More precisely, if $u$ is an extremal with unit $L^q(\Omega)$ norm and we set
\[
v=\Big(\lambda_{p,q}(\Omega)\Big)^\frac{1}{q-p}\,u,
\]
then $v$ solves \eqref{renato}. Moreover, it minimizes the free energy functional
\[
\mathfrak{F}_{p,q}(\varphi)=\frac{1}{p}\,\int_\Omega |\nabla \varphi|^p\,dx-\frac{1}{q}\,\int_\Omega |\varphi|^q\,dx,\qquad \mbox{ for every } \varphi\in W^{1,p}_0(\Omega),
\]
among all the nontrivial critical points of $\mathfrak{F}_{p, q}$.
\vskip.2cm\noindent
With these clarifications  in mind, the existence of solutions of minimal energy for the following slightly modified equation (here $\lambda>0$)
\[
-\Delta_p v+\lambda\,|v|^{p-2}\,v=|v|^{q-2}\,v,\qquad \text{in}\ \Omega,
\]
in the case of {\it periodic open sets} has been obtained in \cite{LTW} and \cite{Cha} for $p=2$, then generalized in \cite{HuLi} to the case $1<p<N$. The result \cite[Theorem 4.8]{LTW} is obtained through a {\it Struwe--type global compactness result}, giving a precise description of Palais-Smale sequences for the free energy functional
\[
\mathfrak{F}_{2,q,\lambda}(\varphi)=\frac{1}{2}\,\int_\Omega |\nabla \varphi|^2\,dx+\frac{\lambda}{2}\,\int_\Omega |\varphi|^2\,dx-\frac{1}{q}\,\int_\Omega |\varphi|^q\,dx,
\]
naturally associated to the modified equation. We refer to \cite{St} for the original result by Struwe, dealing with the critical case $q=2^*$. 
\par
This same existence result has then be reproved by Chabrowski in \cite[Corollary 1]{Cha}, by using a concentration-compactness principle ``at infinity'', in order to prove existence of solutions for the minimization problem
\[
\inf_{\varphi\in W^{1,2}_0(\Omega)}\left\{\int_\Omega |\nabla\varphi|^2\,dx+\lambda\,\int_\Omega |\varphi|^2\,dx\, :\, \|\varphi\|_{L^q(\Omega)}=1\right\}.
\]
The extension to the case $1<p<N$ (see \cite[Theorem 4]{HuLi}) relies on the same arguments by Chabrowski, see also Smets' paper \cite{Sm} for a similar extension. We point out that, in all these results, the assumption $\lambda>0$ can not be removed\footnote{This can be easily seen by observing that these results are stated to hold for $\Omega=\mathbb{R}^N$, as well, where obviously $\lambda_{p,q}$ is zero and can not be attained. In \cite{Cha} the importance of the assumption $\lambda>0$ can be detected in the proof of Theorem 2 there, at beginning of page 506, in the argument used to exclude that the limit of a minimizing sequence identically vanishes. Indeed, the author applies a Poincar\'e-Sobolev inequality of the type 
\[
\|\varphi\|^2_{L^q(E)}\lesssim \|\nabla \varphi\|_{L^2(E)}+\lambda\,\|\varphi\|_{L^2(E)},
\]
for bounded regular sets $E\subseteq\mathbb{R}^N$ and functions {\it not vanishing at the boundary}. For $\lambda=0$ this can not hold. The same argument is at the bottom of \cite[page 69]{HuLi}, for the case $1<p<N$.}.
\vskip.2cm\noindent
In order to highlight the main differences with these results, we first point out that we cover the whole range $1<p<\infty$ and $q>p$ verifying \eqref{exponents}, at the same time. Thus, for example, we can explicitly treat the ``endpoint'' case of Morrey--type $\lambda_{p,\infty}$, occurring for $p>N$. In the case of open bounded sets, some interesting studies on this constant have been pursued in \cite{EP, HL}.
\par
More interestingly, our existence proof is elementary and does not rely neither on concentration-compactness arguments, nor on Struwe's technique. As in our previous paper \cite{BraBriPri}, we will rather adapt to our problem a technique used by Lieb in \cite{LiHLS}, in order to show existence of extremals for the Hardy-Littlewood-Sobolev inequality. The crucial step is to obtain a result in the vein of \cite[Lemma 2.7]{LiHLS}. In a nutshell, the program goes as follows: if we can construct a minimizing sequence $\{u_n\}_{n\in\mathbb{N}}$ for $\lambda_{p,q}(\Omega)$ such that:
\begin{itemize}
\item[(i)] $\{u_n\}_{n\in\mathbb{N}}$ converges almost everywhere to $u$;
\vskip.2cm
\item[(ii)] $\{\nabla u_n\}_{n\in\mathbb{N}}$ converges almost everywhere to $\nabla u$;
\vskip.2cm
\item[(iii)] the limit $u$ is not trivial;
\end{itemize}
then we can infer existence of an extremal. We remark that the subadditivity of the function $t\mapsto t^{p/q}$ enters as a crucial ingredient of the proof of this implication (here our standing assumption $q>p$ is important). 
\par
It is well-known that the almost everywhere convergence of $u_n$ can be easily inferred: it is sufficient to use the compactness of the {\it local} Sobolev embeddings, together with a diagonal argument. On the contrary, property (ii) {\it can not} be expected to hold in general, for a minimizing sequence. We circumvent this problem by employing the same trick successfully  exploited in \cite{BraBriPri} i.e. we will apply this reasoning not to any minimizing sequence, but to a {\it particular one}, which comes from a minimization problem containing an additional confining term. Thus, the relevant optimality condition assures the required convergence, as the small confinement parameter goes to $0$.
\par
However, the most delicate point is to get property (iii): here, we will once more take advantage of the fact that $q$ is super-homogeneous and adapt to our setting the classical ``weak compactness up to translations'' result, again proved by Lieb in the case of the whole $\mathbb{R}^N$ (see \cite[Lemma 6]{Li}). In proving this kind of result for our periodically perforated sets, we will give a different proof of the celebrated result \cite[Corollary 4]{Li}, yet again due to Lieb, asserting that:\vskip.2cm 
{\it if $\Omega\subseteq\mathbb{R}^N$ has a ``small'' Poincar\'e constant, then it contains a ``big'' portion of a ``large'' ball}. 
\vskip.2cm\noindent
The proof will rely on some capacitary--type techniques pioneered by Maz'ya: we think that this part of the paper (contained in Section \ref{sec:3}) is interesting in itself.

\subsection{Plan of the paper}

In Section \ref{sec:2} we introduce the basic notation needed in the paper and show a few preliminary results. We also discuss some properties of a periodically perforated open set $\Omega$ as in Definition \ref{defi:apps_intro}. In particular, we prove that, for such a set,  the  condition  $\mathrm{cap}_p(K; Q_1)>0$  is equivalent to having  $\lambda_p(\Omega)>0$. As announced above, Section \ref{sec:3} contains a novel proof of \cite[Corollary 4]{Li}, which exploits the concept of {\it capacitary inradius}, introduced by Maz'ya. 
\par
Then, Section \ref{sec:4} deals with a general compactness result which is suitable to our periodic setting. This is analogous to \cite[Lemma 6]{Li}, the latter dealing with the case $W^{1,p}(\mathbb{R}^N)$. The last ingredient to apply the aforementioned Lieb's scheme  is discussed in Section \ref{sec:5}: there, we construct an {\it ad hoc} minimizing sequence,  whose elements  are   minimizers of suitably penalized problems. 
\par
The  main results of this paper are given in Section \ref{sec:6} and  in Section \ref{sec:7}. In Section \ref{sec:8},  we discuss the different behavior, in terms of  existence of extremals,   of the periodic open set  \eqref{forellato} when we change the radius of a single hole. A technical fact concerning Poincar\'e-Sobolev constants ``at infinity'' is the content of Appendix \ref{sec:A}, which closes the paper.

\begin{ack}
We thank Francesco Bozzola for some conversations on the contents of Section \ref{sec:3}.
L.\,Briani and F.\, Prinari are both members of the {\it Gruppo Nazionale per l'Analisi Matematica, la Probabilit\`a
e le loro Applicazioni} (GNAMPA) of the Istituto Nazionale di Alta Matematica (INdAM). F.\,Prinari gratefully acknowledges the financial support of the project GNAMPA 2025  ``Ottimizzazione Spettrale, Geometrica e Funzionale" ({\tt CUP E5324001950001}).  
\par
The research of L.~Briani was supported by the DFG through the Emmy Noether Programme (project number 509436910).  
\end{ack}

\section{Preliminaries}
\label{sec:2}
\subsection{Notation}
In this paper, we will always assume that the dimension $N$ of the ambient space is $N\ge 2$. For $x_0\in\mathbb{R}^N$ and $r>0$, we will set
\[
B_r(x_0)=\Big\{x\in\mathbb{R}^N\,:\, |x-x_0|<r\Big\}.
\]
When the ball is centered at the origin, we will omit to indicate the center and simply write $B_r$. 
Analogously, we set
\[
Q_r=(-r,r)\times\dots\times(-r,r)=(-r,r)^N,
\]
and
\[
Q_r(x_0)=Q_r+x_0=\prod_{i=1}^N (x_0^i-r,x_0^i+r),\qquad \text{for}\ x_0=(x_0^1,\dots,x_0^N)\in\mathbb{R}^N.
\]
We recall that, fixed $1\le p<\infty$, an open set $E\subseteq\mathbb{R}^{N}$ and a compact subset $K\subseteq E$, the {\it $p-$capacity of $K$ relative to $E$} is defined as 
\[
\mathrm{cap}_p(K;E)=\inf_{u\in C^\infty_0(E)}\left\{\int_{E}|\nabla u|^p\,dx\,:\, u\ge 1\ \text{on}\ K\right\}.
\]
We refer to \cite[Chapter 2, Section 2]{Maz} and \cite[Chapter 13]{Maz} for the properties of $p-$capacity.
\subsection{Two technical results}
We start with the following simple result, which gives a rough estimate of how Poincar\'e-Sobolev constants are affected by an affine transformation.
\begin{lm}
\label{lm:affine}
Let $E\subseteq\mathbb{R}^N$ be an open set. Let $T:\mathbb{R}^N\to\mathbb{R}^N$ be an invertible affine map, i.e. there exist an invertible matrix $A\in\mathcal{M}_N(\mathbb{R})$ and a vector $\mathbf{b}\in\mathbb{R}^N$ such that
\[
T(x)=A\cdot x+\mathbf{b},\qquad \text{for every} \ x\in\mathbb{R}^N.
\]
Then for every $1\le p<\infty$ and $q\ge 1$ satisfying \eqref{exponents}
we have\footnote{When $q=\infty$, we agree that $(q-p)/q=1$.} 
\[
\frac{|\det A|^\frac{q-p}{q}}{\ell^\frac{p}{2}}\,\lambda_{p,q}(E)\ge \lambda_{p,q}(T(E))\ge \frac{|\det A|^\frac{q-p}{q}}{L^\frac{p}{2}}\,\lambda_{p,q}(E),
\]
where $L\ge \ell>0$ are the maximal and minimal eigenvalues of the symmetric positive definite matrix $A^{\rm T}\cdot A$.
\end{lm}
\begin{proof}
The proof simply follows by using a change of variable and the following facts:
\[
|A\cdot z|^2=\langle A\cdot z,A\cdot z\rangle=\langle A^{\rm T}\cdot A\cdot z,z\rangle\ge \ell\,|z|^2,\qquad \text{for every}\ z\in\mathbb{R}^N
\]
and
\[
|A^{-1}\cdot z|^2=\langle A^{-1}\cdot z,A^{-1}\cdot z\rangle=\langle (A^{-1})^{\rm T}\,A^{-1}\cdot z,z\rangle\ge \frac{1}{L}\,|z|^2,\qquad \text{for every}\ z\in\mathbb{R}^N.
\]
The details are left to the reader.
\end{proof}
The next result is specific of the case $q>p$. It concerns the Poincar\'e-Sobolev constants of disjoint sets and can be seen as a refinement of \cite[Proposition 2.6]{BraFra}.  
\begin{lm}
\label{lm:disgiunti} 
Let $1<p<\infty$ and assume that  $q>p$ satisfies \eqref{exponents}. Let $ \{\Omega_i\}_{i\in\mathbb{N}}$ be a family of open subsets of $\mathbb{R}^N$ such that $\Omega_i\cap \Omega_j=\emptyset$, for every $i\not=j$. We set
\[
\Omega=\bigcup_{i\in\mathbb{N}}\Omega_i,
\]
and suppose that there exists an extremal $u\in W^{1,p}_0(\Omega)$ for $\lambda_{p,q}(\Omega)$, with $\|u\|_{L^q(\Omega)}=1$. Then there exists $i_0\in\mathbb{N}$ such that 
\[
\|u\|_{L^q(\Omega_{i_0})}=1\qquad \text{and}\qquad \|u\|_{L^q(\Omega_i)}=0,\quad \text{for}\ i\not=i_0.
\]
\end{lm}
\begin{proof}
We distinguish two cases: either $q<\infty$ or $q=\infty$.
\vskip.2cm\noindent
{\it Case $q<\infty$}. Since $u\not\equiv 0$ we have that 
\[
\int_{\Omega_{i_0}}|u|^q\,dx>0,
\]
for some $i_0\in\mathbb{N}$. It is enough to prove that $\|u\|_{L^q(\Omega_{i_0})}=1$. We argue by contradiction and suppose that the latter does not hold true. This means that 
\begin{equation}\label{entrambipositivi}
\int_{\Omega_{i_0}}|u|^q\,dx>0\qquad \text{and}\qquad\sum_{j\neq i_0}\int_{\Omega_j}|u|^q\,dx>0.
\end{equation}
Since $q>p$, by strict subadditivity of the map $\tau\mapsto \tau^{p/q}$ we have the following inequality
\[
\begin{split}
1=\left(\int_{\Omega}|u|^q\,dx\right)^{\frac{p}{q}}&=\left(\sum_{j\neq i_0}\int_{\Omega_j}|u|^q\,dx+\int_{\Omega_{i_0}}|u|^q\,dx\right)^{\frac{p}{q}}\\
&<\left(\sum_{j\neq i_0}\int_{\Omega_j}|u|^q\,dx\right)^{\frac{p}{q}}+\left(\int_{\Omega_i}|u|^q\,dx\right)^{\frac{p}{q}}\le \sum_{i\in\mathbb{N}}\left(\int_{\Omega_j}|u|^q\,dx\right)^{\frac{p}{q}}.
\end{split}
\]
The strict inequality holds thanks to \eqref{entrambipositivi}. We thus have obtained that
\[
1<\sum_{i\in\mathbb{N}}\left(\int_{\Omega_j}|u|^q\,dx\right)^{\frac{p}{q}}.
\]
By using this fact and the definition of $\lambda_{p,q}(\Omega_i)$, we get
\[
\begin{split}
\lambda_{p,q}(\Omega)&=\sum_{i\in\mathbb{N}}\int_{\Omega_i}|\nabla u|^{p}\,dx\ge\sum_{i\in\mathbb{N}} \lambda_{p,q}(\Omega_i)\left(\int_{\Omega_i}|u|^q\,dx\right)^{\frac{p}{q}}>\inf_{i\in\mathbb{N}}\lambda_{p,q}(\Omega_i).
\end{split}
\]
On the other hand, by monotonicity we have
\[
\lambda_{p,q}(\Omega)\le\lambda_{p,q}(\Omega_i),\qquad \text{for every}\ i\in\mathbb{N}.
\]
Thus, we get a contradiction and the result is proved when $q<\infty$.
\vskip.2cm\noindent
{\it Case $q=\infty$}. This case can occur only for $p>N$. Thanks to \cite[Lemma A.1]{BraBriPri}, there exists $x_0\in \Omega$ such that $|u(x_0)|=\|u\|_{L^\infty(\Omega)}=1$. In particular, there exists $i_0\in\mathbb{N}$ such that $x_0\in \Omega_{i_0}$. Accordingly, we get
\[
\|u\|_{L^\infty(\Omega_{i_0})}=1.
\]
Let us assume that there exists $j_0\not=i_0$ such that $\|u\|_{L^\infty(\Omega_{j_0})}>0$. Thus, we get
\[
\begin{split}
\lambda_{p,\infty}(\Omega)=\sum_{i\in\mathbb{N}}\int_{\Omega_i} |\nabla u|^p\,dx&\ge \int_{\Omega_{i_0}} |\nabla u|^p\,dx+\int_{\Omega_{j_0}} |\nabla u|^p\,dx\\
&\ge \lambda_{p,\infty}(\Omega_{i_0})\,\|u\|_{L^\infty(\Omega_{i_0})}^p+ \lambda_{p,\infty}(\Omega_{j_0})\,\|u\|_{L^\infty(\Omega_{j_0})}^p\\
&>\min\Big\{\lambda_{p,\infty}(\Omega_{i_0}),\lambda_{p,\infty}(\Omega_{j_0})\Big\}.
\end{split}
\]
As in the previous case, we get a contradiction with the fact that
\[
\lambda_{p,\infty}(\Omega)\le \min\Big\{\lambda_{p,\infty}(\Omega_{i_0}),\lambda_{p,\infty}(\Omega_{j_0})\Big\}.
\]
Thus, we can conclude that $\|u\|_{L^\infty(\Omega_{j})}=0$ for every $j\not= i_0$.
\end{proof}

\subsection{Periodically perforated sets}
We give some properties of periodically perforated sets, introduced in Definition \ref{defi:apps_intro}.
In what follows, for an open set $\Omega\subseteq\mathbb{R}^N$, we will indicate by $r_\Omega$ its {\it inradius}, defined by
\[
r_\Omega=\sup\{r>0\, :\, \exists\ B_r(x_0)\subseteq\Omega\}.
\]
\begin{lm}\label{lm:inradius}
Let $\Omega\subseteq\mathbb{R}^N$ be a $\mathbf{t}-$periodically perforated open set, generated by $K$. Then, we have 
\[
r_{\Omega}\leq \max  \{ t_1,t_2,\cdots, t_N\}\, \frac{\sqrt{N}}{2} .
\]
\end{lm}
\begin{proof}
By definition, we see that the inradius is monotone non-decreasing with respect to the set inclusion. In other words, if $E_1\subseteq E_2$ are two open sets, then we have $r_{E_1}\le r_{E_2}$. Moreover, the inradius is clearly invariant by translations, i.e. for every $x_0\in\mathbb{R}^N$ we have $r_{E+x_0}=r_E$.
\par
We thus fix $x_0\in K$, then by construction we have
\[
\Omega\subseteq\mathbb{R}^N\setminus\left(\bigcup_{\mathbf{i}\in\mathbb{Z}^N}D_{\mathbf{t}}(\mathbf{i}+x_0)\right)=\left(\mathbb{R}^N\setminus\left(\bigcup_{\mathbf{i}\in\mathbb{Z}^N}D_{\mathbf{t}}(\mathbf{i})\right)\right)+D_\mathbf{t}(x_0)=:\widetilde{\Omega}+D_\mathbf{t}(x_0).
\]
From the two aforementioned properties, we then get $r_\Omega\le r_{\widetilde\Omega}$. It is now not difficult to see that 
\[
r_{\widetilde\Omega}\le \max  \{ t_1,t_2,\cdots, t_N\}\, \frac{\sqrt{N}}{2},
\]
thus concluding.
\end{proof}
\begin{rem}\label{controcroci}
Observe that the previous upper bound is independent of the generating set $K$ and optimal. Indeed, it becomes an identity if we take 
\[
\Omega=\mathbb{R}^N\setminus \mathbb{Z}^N,
\]
which corresponds to the choices $\mathbf{t}=(1,\dots,1)$ and $K=\{0\}$.
\par
On the contrary, it is not possible to give a {\it lower bound} of the same type. 
Indeed, for every $0<\varepsilon<1/2$, let us take $\mathbf{t}=(1,\dots,1)$ and $K=\overline{Q_{1/2-\varepsilon}}$. Accordingly, the corresponding periodically perforated open set is given by
\[
\Omega_\varepsilon=\mathbb{R}^N\setminus\left(\bigcup_{\mathbf{i}\in\mathbb{Z}^N}\overline{Q_{1/2-\varepsilon}(\mathbf{i})}\right).
\]
We have that $r_{\Omega_{\epsilon}}=\varepsilon\, \sqrt{N}$, which clearly vanishes in the limit as $\varepsilon$ goes to zero.
\end{rem}
For our open sets, the sharp condition for the validity of the Poincar\'e inequality is simply expressed in terms of a suitable capacitary requirement on the generating set $K$. This is the content of the following lemma. Even if not needed in this paper, we also include the case $p=1$, for completeness.
\begin{lm}
\label{lm:nonbanale}
Let $1\le p<\infty$ and let $\Omega\subseteq\mathbb{R}^N$ be a $\mathbf{t}-$periodically perforated open set, generated by $K$. We have that
\[
\lambda_p(\Omega)>0\qquad \Longleftrightarrow\qquad \mathrm{cap}_p(K; Q_1)>0.
\]
Moreover, there holds
\begin{equation}
\label{louerbaund}
\lambda_p(\Omega)\ge c_{N,p}\, \left(\frac{1}{\max\{t_1,\dots,t_N\}}\right)^p\,\mathrm{cap}_p(K;Q_1).
\end{equation}
\end{lm}
\begin{proof}
We prove at first the estimate \eqref{louerbaund}. This would establish the validity of the implication 
\[
\mathrm{cap}_p(K; Q_1)>0\qquad \Longrightarrow\qquad \lambda_p(\Omega)>0.
\]
We use Lemma \ref{lm:affine} with $T=D_\mathbf{t}$ and
\begin{equation}
\label{E}
E=\mathbb{R}^N\setminus\left(\bigcup_{\mathbf{i}\in\mathbb{Z}^N}(\mathbf{i}+K)\right).
\end{equation}
Since $D_\mathbf{t}(E)=\Omega$, this gives immediately
\[
\lambda_p(\Omega)=\lambda_p(D_\mathbf{t}(E))\ge \left(\frac{1}{\max\{t_1,\dots,t_N\}}\right)^p\,\lambda_p(E).
\]
It is now not difficult to see that for every $\varphi\in C^\infty_0(E)$, we have
\[
\int_E |\nabla \varphi|^p\,dx=\sum_{\mathbf{i}\in \mathbb{Z}^N} \int_{\mathbf{i}+Q_{1/2}} |\nabla \varphi|^p\,dx\ge c_{N,p}\,\mathrm{cap}_p(K;Q_1)\,\sum_{\mathbf{i}\in \mathbb{Z}^N} \int_{\mathbf{i}+Q_{1/2}} |\varphi|^p\,dx,
\]
thanks to the Maz'ya-Poincar\'e inequality of \cite[Chapter 14, Theorem 14.1.2]{Maz}. By arbitrariness of $\varphi\in C^\infty_0(E)$, this gives
\[
\lambda_p(E)\ge c_{N,p}\,\mathrm{cap}_p(K;Q_1),
\]
and thus the desired conclusion follows.
\vskip.2cm\noindent
We now suppose that $\lambda_p(\Omega)>0$. 
We argue by contradiction and assume that $\mathrm{cap}_p(K;Q_1)=0$. We first observe that if $p>N$, then we would get the contradiction $K=\emptyset$, since every non-empty compact set has positive $p-$capacity in this case. 
\par
We can thus assume that $1<p\le N$. As above, we consider the set $E$ defined by \eqref{E} and $T=D_\mathbf{t}$.
By Lemma \ref{lm:affine}, we get that $\lambda_p(E)>0$, as well. 
Thanks to the assumption on $K$ and the monotonicity of the relative $p-$capacity (see \cite[page 142]{Maz}), we have 
\[
\mathrm{cap}_p(K;Q_M)=0,\qquad \text{for every}\ M\ge 1,
\]
as well. We choose $k\in\mathbb{N}\setminus\{0\}$ and use this fact with $M=k+1$. We also set
\begin{equation}
\label{Ztroncato}
\mathbb{Z}^N_k=\Big\{\mathbf{i}=(i_1,\dots,i_N)\in\mathbb{Z}^N\, :\, |\mathbf{i}|_{\ell^\infty}:=\max_{m=1,\dots,N}|i_m|\le k\Big\},
\end{equation}
and define
\[
K_k=\bigcup_{\mathbf{i}\in\mathbb{Z}^N_k} (\mathbf{i}+K).
\]
Observe that for every $\mathbf{i}\in \mathbb{Z}^N_k$, we have that $\mathbf{i}+K$ is a compact subset of $\mathbf{i}+Q_1=Q_1(\mathbf{i})$, thus by translation invariance of the $p-$capacity we still have
\[
\mathrm{cap}_p(\mathbf{i}+K;Q_1(\mathbf{i}))=\mathrm{cap}_p(K;Q_1)=0.
\] 
By observing that $Q_1(\mathbf{i})\subseteq Q_{k+1}$ for every $\mathbf{i}\in \mathbb{Z}^N_k$, we get as before
\[
\mathrm{cap}_p(\mathbf{i}+K;Q_{k+1})=0.
\]
By the subadditivity of the relative $p-$capacity (see \cite[page 143]{Maz}), this in turn implies that
\[
\mathrm{cap}_p(K_k;Q_{k+1})=0,
\]
as well.
Thus, for every $\varepsilon>0$ there exists $\varphi_\varepsilon\in C^\infty_0(Q_{k+1})$ such that
\[
\int_{Q_{k+1}} |\nabla \varphi_\varepsilon|^p\,dx<\varepsilon,\qquad 0\le \varphi_\varepsilon\le 1,\qquad \varphi_\varepsilon\equiv 1\ \text{on}\ K_k.
\]
We consider $\varphi_\varepsilon$ to be extended by $0$, outside $Q_{k+1}$.
Observe that by the Markov-Chebychev and Poincar\'e inequalities, we have
\[
\begin{split}
\left|\left\{x\in Q_{k-1/2}\, :\, \varphi_\varepsilon(x)>\frac{1}{2}\right\}\right|&\le 2^p\,\int_{Q_{k-1/2}} |\varphi_\varepsilon|^p\,dx\\
&\le 2^p\,\int_{Q_{k+1}}|\varphi_\varepsilon|^p\,dx\le \frac{2^p}{\lambda_p(Q_{k+1})}\,\int_{Q_{k+1}} |\nabla \varphi_\varepsilon|^p\,dx<\frac{2^p\,\varepsilon}{\lambda_p(Q_{k+1})}.
\end{split}
\]
Thus, by choosing 
\[
\varepsilon_k:=\frac{|Q_{k-1/2}|}{2^{p+1}}\,\lambda_p(Q_{k+1})=\frac{\left(2\,k-1\right)^{N}}{2^{p+1}}\,\frac{\lambda_p(Q_1)}{(k+1)^p},
\]
we get that 
\begin{equation}
\label{minorato}
\begin{split}
\int_{Q_{k-\frac{1}{2}}} |1-\varphi_\varepsilon|^p\,dx&\ge \frac{1}{2^p}\,\left|\left\{x\in Q_{k-1/2}\, :\, 0\le \varphi_\varepsilon(x)\le \frac{1}{2}\right\}\right|\\
&\ge \frac{|Q_{k-1/2}|}{2^{p+1}},\qquad \text{for every}\ 0<\varepsilon\le \varepsilon_k.
\end{split}
\end{equation}
We now take a cut-off function $\eta_k\in C^\infty_0(Q_{k+1/2})$ such that 
\[
\eta_k\equiv 1\ \text{on}\ Q_{k-\frac{1}{2}},\qquad 0\le \eta_k\le 1,\qquad |\nabla \eta_k|\le C,
\]
for a universal constant $C>0$. Finally, we define
\[
U_{\varepsilon,k}(x)=\eta_k(x)\,(1-\varphi_\varepsilon(x)),\qquad \text{for}\ 0<\varepsilon\le \varepsilon_k,
\]
which belongs to $W^{1,p}_0(E)$. Indeed, this is a smooth function with finite $W^{1,p}$ norm, vanishing at the boundary $\partial E$. Then it is sufficient to appeal to \cite[Theorem 9.17, Remark 1]{Brezis}. Thus, we get
\[
\begin{split}
\lambda_p(E)&\le \frac{\displaystyle\int_{Q_{k+1/2}} |\nabla U_{\varepsilon,k}|^p\,dx}{\displaystyle\int_{Q_{k+1/2}}|U_{\varepsilon,k}|^p\,dx}\\
&\le 2^{p-1}\,\frac{\displaystyle\int_{Q_{k+1/2}} \left|\nabla \varphi_\varepsilon\right|^p\,dx}{\displaystyle\int_{Q_{k-1/2}}\left|1-\varphi_\varepsilon\right|^p\,dx}+2^{p-1}\,\frac{\displaystyle\int_{Q_{k+1/2}\setminus Q_{k-1/2}} |\nabla\eta_k|^p\, \left|1-\varphi_\varepsilon\right|^p\,dx}{\displaystyle\int_{Q_{k-1/2}}\left|1-\varphi_\varepsilon\right|^p\,dx}.
\end{split}
\]
We can estimate the denominator thanks to \eqref{minorato}, while in the second term we use that $|1-\varphi_\varepsilon|\le 1$. By recalling the estimate on the $p-$Dirichlet integral of $\varphi_\varepsilon$, we obtain
\[
\begin{split}
\lambda_p(E)&\le \frac{4^p\,\varepsilon}{|Q_{k-1/2}|}+\frac{4^p\,C^p\,|Q_{k+1/2}\setminus Q_{k-1/2}|}{|Q_{k-1/2}|}.
\end{split}
\]
Observe that, by its definition, we have that $\varepsilon_k$ stays uniformly bounded from below, as $k$ goes to $\infty$. We can thus, for $k$ large enough, use the previous estimate with a fixed $\varepsilon>0$, independent of $k$. By taking the limit as $k$ goes to $\infty$ and observing that
\[
\lim_{k\to\infty} \frac{|Q_{k+1/2}\setminus Q_{k-1/2}|}{|Q_{k-1/2}|}=\lim_{k\to\infty} \frac{(2\,k+1)^N-(2\,k-1)^N}{(2\,k-1)^N}=0,
\]
we reach the desired contradiction $\lambda_p(E)=0$. This concludes the proof.
\end{proof}

\begin{rem}
Without further assumptions on the generating set $K$, it is not possible to revert the estimate \eqref{louerbaund}. Indeed, take for example 
\[
K_\varepsilon=\overline{Q_{1/2}}\setminus B_\varepsilon,\qquad \text{for}\ 0<\varepsilon<\frac{1}{2}.
\]
Accordingly, the periodically perforated set generated by $K_\varepsilon$ is simply given by
\[
\Omega_\varepsilon=\bigcup_{\mathbf{i}\in\mathbb{Z}^N} B_\varepsilon(\mathbf{i}),
\]
i.e. this is a disjoint union of open balls, each having radius $\varepsilon$. Thanks to the properties of $\lambda_p$, we thus get
\[
\lambda_p(\Omega_\varepsilon)=\lambda_p(B_\varepsilon)=\varepsilon^{-p}\,\lambda_p(B_1),
\]
which blows-up, as $\varepsilon$ goes to $0$. On the other hand, we clearly have
\[
\mathrm{cap}_p(K_\varepsilon;Q_1)\le \mathrm{cap}_p(\overline{Q_{1/2}};Q_1),
\]
thus the relative $p-$capacity of $K_\varepsilon$ stays uniformly bounded from above.
\par
The reader may be unsatisfied with this example, 
as each open set $\Omega_\varepsilon$ has infinitely many connected components. In this sense, a more interesting counterexample is provided by the sets 
\[
\Omega_\varepsilon=\mathbb{R}^N\setminus\left(\bigcup_{\mathbf{i}\in\mathbb{Z}^N}\overline{Q_{1/2-\varepsilon}(\mathbf{i})}\right),
\]
considered in Remark \ref{controcroci}. Here as well, we have 
\[
\lim_{\varepsilon\searrow 0}\lambda_p(\Omega_\varepsilon)=+\infty\qquad \text{and}\qquad \sup_{0<\varepsilon<1/2} \mathrm{cap}_p(K_\varepsilon;Q_1)<+\infty,\ \text{where}\ K_\varepsilon:=\overline{Q_{1/2-\varepsilon}}.
\]
In this case, in order to show that $\lambda_{p}(\Omega_\varepsilon)$ blows up as $\varepsilon$ approaches $0$, we have to appeal to a more sophisticated estimate. For instance, we can use the results contained in \cite[Section 5]{BozBra2}. Indeed, with the terminology of that paper, it is easy to see that the family of sets $\Omega_\varepsilon$ has a {\it measure density index} uniformly bounded from below. Thus, by\footnote{More precisely, we can apply \cite[Corollary 5.6]{BozBra2} by choosing $p=q$, $r_0=\varepsilon$, $\theta=2^{-N}$ and $t_0=0$.} \cite[Corollary 5.6]{BozBra2} this implies that 
\[
c_{N,p}\,\left(\frac{1}{r_{\Omega_\varepsilon}}\right)^p\le\lambda_p(\Omega_\varepsilon).
\]
\end{rem}

\section{A Maz'ya-Lieb--type estimate}
\label{sec:3}
In this section we will give a different proof of the following fact, due to Lieb (see \cite[Corollary 4]{Li}): if an open set $\Omega$ has a ``small'' Poincar\'e constant, then it should contain a ``big'' portion of a ``large'' ball. 
This will be useful in order to get a compactness lemma {\it \`a la} Lieb, in the next section. Throughout this section, for completeness we will include the limit case $p=1$, as well.
\vskip.2cm\noindent
We need at first the concept of {\it negligible set} in the sense of Molchanov, see \cite[Chapters 14 \& 18]{Maz}. This permits to introduce a suitable capacitary variant of the inradius, as in \cite{BozBra2, BozBra,MaSh}.
\begin{definition}
Let $1\le p<\infty$ and $0<\gamma<1$, we say that a compact set $\Sigma\subseteq \overline{B_r(x_0)}$ 
is {\it $(p,\gamma)-$negligible} if 
\[
\mathrm{cap}_p(\Sigma;B_{2r}(x_0))\le \gamma\,\mathrm{cap}_p\left(\overline{B_r(x_0)};B_{2r}(x_0)\right).
\]
Accordingly, we consider the {\it capacitary inradius of $\Omega$}, defined as follows
\[
R_{p,\gamma}(\Omega):=\sup\Big\{r>0\, :\, \exists\, x_0\in\mathbb{R}^N\ \text{such that}\ \overline{B_r(x_0)}\setminus\Omega\ \text{is $(p,\gamma)-$negligible}\Big\}.
\]
From its definition, we can see that
\[
r_\Omega\le R_{p,\gamma}(\Omega),\ \text{for every}\ 0<\gamma<1,\qquad\text{and}\qquad \gamma \mapsto R_{p,\gamma}(\Omega) \ \text{is monotone non-decreasing}.
\]
\end{definition}

\begin{lm}
\label{lm:misura}
Let $1\le p\le N$ and $0<\gamma<1$. Let $\Omega\subseteq\mathbb{R}^N$ be an open set. If the set $\overline{B_r(x_0)}\setminus\Omega$ is $(p,\gamma)-$negligible, then
\[
|B_r(x_0)\cap \Omega|\ge \left(1-\Gamma_{N,p}\,\gamma\right)\,|B_r(x_0)|,
\]
where the constant $\Gamma_{N,p}\ge1$ is given by
\[
\Gamma_{N,p}=\frac{\mathrm{cap}_p\left(\overline{B_1};B_{2}\right)}{|B_1|\,\lambda_p(B_2)}.
\]
\end{lm}
\begin{proof}
We estimate the relative $p-$capacity of $\overline{B_r(x_0)}\setminus\Omega$ by \cite[Corollary 2.3.4]{Maz}. This gives
\[
\begin{split}
\mathrm{cap}_p\left(\overline{B_r(x_0)}\setminus\Omega;B_{2r}(x_0)\right)&\ge \lambda_p(B_{2r}(x_0))\,\left|\overline{B_r(x_0)}\setminus\Omega\right|.
\end{split}
\]
On the other hand, by assumption of $(p,\gamma)-$negligibility we have
\[
\begin{split}
\mathrm{cap}_p\left(\overline{B_r(x_0)}\setminus\Omega;B_{2r}(x_0)\right)&\le \gamma\,\mathrm{cap}_p\left(\overline{B_r(x_0)};B_{2r}(x_0)\right)\\
&=\gamma\,r^{N-p}\,\mathrm{cap}_p\left(\overline{B_1};B_{2}\right).
\end{split}
\]
By joining the last two equations in display and simplifying the common terms, we thus obtain
\[
\lambda_p(B_2)\,\left|\overline{B_r(x_0)}\setminus\Omega\right|\le \gamma\,r^{N}\,\mathrm{cap}_p\left(\overline{B_1};B_{2}\right).
\]
We used the scaling properties of $\lambda_p$ to ensure that $\lambda_p(B_{2r}(x_0))=r^{-p}\,\lambda_p(B_2)$.
With simple manipulations, we see that this is equivalent to
\[
\left|\overline{B_r(x_0)}\setminus\Omega\right|\le \frac{1}{\lambda_p(B_2)}\,\frac{\mathrm{cap}_p\left(\overline{B_1};B_{2}\right)}{|B_1|}\,\gamma\,|B_{r}(x_0)|=:\Gamma_{N,p}\,\gamma\,|B_{r}(x_0)|.
\]
We observe that 
\[
|B_1|\,\lambda_p(B_2)\le \mathrm{cap}_p(\overline{B_1};B_2),
\]
again by \cite[Corollary 2.3.4]{Maz}, thus in particular we have that $\Gamma_{N,p}\ge 1$.
This is enough to conclude: indeed, we can now infer
\[
\begin{split}
\left|\overline{B_r(x_0)}\cap \Omega\right|=|B_r(x_0)|-\left|\overline{B_r(x_0)}\setminus\Omega\right|&\ge |B_r(x_0)|-\Gamma_{N,p}\,\gamma\,|B_{r}(x_0)|\\
&=\left(1-\Gamma_{N,p}\,\gamma\right)\,|B_r(x_0)|,
\end{split}
\]
as desired.
\end{proof}
The next one is the main result of this section. This provides a different proof of \cite[Corollary 4]{Li} by Lieb: in the proof, we will use an idea of Maz'ya and Shubin (see \cite[Section 5]{MaSh}), combined with the recent result \cite[Main Theorem]{BozBra} by Bozzola and the first author.
\begin{teo}
\label{teo:liebtype}
Let $\Omega\subseteq\mathbb{R}^N$ be an open set such that $\lambda_p(\Omega)>0$. 
\begin{itemize}
\item Let $1\le p\le N$. Then for every $0<\beta<1$ there exists a constant $\alpha_{N,p,\beta}>0$ and a ball $B_r(x_0)$ with radius 
\[
r\ge \frac{\alpha_{N,p,\beta}}{\Big(\lambda_p(\Omega)\Big)^\frac{1}{p}},
\] 
such that 
\[
|B_r(x_0)\cap \Omega|\ge \beta\,|B_r(x_0)|.
\] 
The constant $\alpha_{N,p,\beta}$ is such that:
\[
\alpha_{N,p,\beta}\sim (1-\beta)^\frac{1}{p},\qquad \text{as}\ \beta\nearrow 1.
\]
\item Let $p>N$. Then there exists a constant $\alpha_{N,p}>0$ and a ball $B_r(x_0)$ with radius 
\[
r\ge \frac{\alpha_{N,p}}{\Big(\lambda_p(\Omega)\Big)^\frac{1}{p}},
\] 
such that $B_r(x_0)\subseteq \Omega$.
\end{itemize}
\end{teo}
\begin{proof}
We distinguish two cases, either $1\le p\le N$ or $p>N$.
\vskip.2cm\noindent
{\it Case $1\le p\le N$}. By \cite[Main Theorem]{BozBra}, we know that there exists a  constant $\sigma_{N,p}>0$ such that,  for every $0<\gamma<1$, it holds 
\[
\gamma\,\sigma_{N,p}\,\left(\frac{1}{R_{p,\gamma}(\Omega)}\right)^p\le \lambda_p(\Omega).
\]
In particular, by choosing 
\[
\gamma=\gamma_\beta:=\frac{1-\beta}{\Gamma_{N,p}},
\] 
where $\Gamma_{N,p}$ is the same constant as in Lemma \ref{lm:misura},
we get
\[
R_{p,\gamma_\beta}(\Omega)\ge \left(\frac{\sigma_{N,p}\,\gamma_\beta}{\lambda_p(\Omega)}\right)^\frac{1}{p}=:\frac{2\,\alpha_{N,p,\beta}}{\Big(\lambda_p(\Omega)\Big)^\frac{1}{p}}.
\]
By definition of capacitary inradius, there exists a ball $B_{r}(x_0)$ with
\[
\frac{R_{p,\gamma_\beta}(\Omega)}{2}<r<R_{p,\gamma_\beta}(\Omega),
\]
such that the set $\overline{B_r(x_0)}\setminus \Omega$ is $(p,\gamma_\beta)-$negligible. By using Lemma \ref{lm:misura} and the choice of $\gamma$, we conclude.
\vskip.2cm\noindent
{\it Case $p>N$.} This case is simpler. Indeed, we know that for $p>N$ there holds
\[
\frac{C_{N,p}}{r_\Omega^p}\le \lambda_p(\Omega),
\]
see, for example, \cite[Theorem 1.4.1]{Po1}, \cite[Theorem 1.1]{Vit} and, more recently, \cite[Theorem 1.3]{BozBra0} and \cite[Corollary 5.9]{BraPriZag}. Moreover, by definition of inradius, there exists a ball $B_r(x_0)$ with
\[
\frac{r_\Omega}{2}<r<r_\Omega,
\]
such that $B_r(x_0)\subseteq\Omega$.
By choosing this time
\[
\alpha_{N,p}=\frac{1}{2}\,\left(C_{N,p}\right)^\frac{1}{p},
\]
we conclude.
\end{proof}

\section{A compactness result}
\label{sec:4}

We will use Theorem \ref{teo:liebtype} to deduce a weak compactness result for certain sequences of functions in $W^{1,p}_0(\Omega)$, when $\Omega$ is a periodically perforated set. The main point is obtaining that the limit function {\it is not trivial}. We adapt the idea of \cite[Lemma 6]{Li}: the latter is concerned with functions in
$W^{1,p}(\mathbb{R}^N)$, now we have to take care of the fact that we work in $W^{1,p}_0(\Omega)$. 
\begin{lm}\label{lm:lieb}
Let $1<p<\infty$ and let $\Omega\subseteq\mathbb{R}^N$ be a $\mathbf{t}-$periodically perforated open set, generated by $K$. Let us suppose that
\[
\mathrm{cap}_p(K; Q_1)>0.
\]
Let  $\{u_n\}_{n\in\mathbb{N}}\subseteq W^{1,p}_0(\Omega)$ be a sequence of non-negative functions such that
\[
\|\nabla u_n\|^p_{L^p(\Omega)}\le C \qquad \text{and}\qquad
|\left\{x\in\Omega\, :\, u_n(x)>\varepsilon\right\}|>\delta,\quad \text{ for all }n\in\mathbb{N},
\]
for some $C,\varepsilon,\delta>0$ independent of $n$.
\par
Then, there exists a sequence of vectors $\{\mathbf{i}_n\}_{n\in\mathbb{N}}\subseteq \mathbb{Z}^N$ such that,
up to a subsequence, the translated sequence $\{u_n(D_{\mathbf{t}}(\mathbf{i}_n)+\cdot)\}_{n\in\mathbb{N}}$ converges weakly in $W^{1,p}(\Omega)$ to a function $u\in W^{1,p}_0(\Omega)\setminus\{0\}$.
\end{lm}
\begin{proof}
We divide the proof in two parts: we first prove the result under the further restriction that each $u_n$ belongs to $C^\infty_0(\Omega)$. Then, in the second part we show how to remove this assumption.
\vskip.2cm\noindent
{\it Part 1: regular functions.} We define at first the following set
\[
A_n=\left\{x\in\Omega\, :\, u_n(x)>\varepsilon/4\right\},\qquad \text{for every}\ n\in\mathbb{N}.
\]
Observe that this is an open set, since we are assuming that each $u_n$ is smooth. We also have $\lambda_p(A_n)>0$, by using that each $u_n$ compactly supported in $\Omega$ and thus $A_n$ is bounded. 
According to Theorem \ref{teo:liebtype} (applied with $\beta=1/2$, in the case $p\le N$), for every $n\in\mathbb{N}$ there exists a radius 
\begin{equation}\label{boundbelow}
r_n\ge \frac{\alpha_{N,p}}{\Big(\lambda_p(A_n)\Big)^\frac{1}{p}},
\end{equation}
and a point $x_n\in\mathbb{R}^N$ such that 
\begin{equation}
\label{mezzino}
|B_{r_n}(x_n)\cap A_n|\ge \frac{1}{2}\,|B_{r_n}(x_n)|.
\end{equation}
We claim that there exist $0<c_1<c_2$ such that 
\begin{equation}
\label{radius_uniform}
c_1\le r_n\le c_2,\qquad \text{for every}\ n\in\mathbb{N}.
\end{equation}
We start with the upper bound: by the Markov-Chebyshev inequality, we have
\[
\begin{split}
|A_n|\le \left(\frac{4}{\varepsilon}\right)^p\,\int_\Omega |u_n|^p\,dx&\le \left(\frac{4}{\varepsilon}\right)^p\,\frac{1}{\lambda_p(\Omega)}\,\int_\Omega |\nabla u_n|^p\,dx\le \left(\frac{4}{\varepsilon}\right)^p\,\frac{C}{\lambda_p(\Omega)}.
\end{split}
\]
Observe that $\lambda_p(\Omega)>0$, thanks to Lemma \ref{lm:nonbanale} and the capacitary assumption on $K$.
Thus, from \eqref{mezzino}, we get in particular
\[
\frac{\omega_N}{2}\,r_n^N\le |A_n|\le \left(\frac{4}{\varepsilon}\right)^p\,\frac{C}{\lambda_p(\Omega)},\qquad \text{for every}\ n\in\mathbb{N}.
\]
In order to bound $r_n$ uniformly from below, in view of \eqref{boundbelow}, it is sufficient to prove that $\lambda_p(A_n)$ is uniformly bounded from above, in terms of the data $C,\varepsilon,\delta$ and $p$, only.  To this aim, we introduce the set
\[
E_n=\left\{x\in\Omega\, :\, u_n(x)>\varepsilon\right\}\subseteq A_n,\qquad \text{for every}\ n\in\mathbb{N}.
\]
We also define $g_n=(u_n-\varepsilon/2)_+$. We notice that $g_n\in W^{1,p}_0(A_n)$: indeed, by construction, $g_n$ is a Lipschitz function, with compact support contained in $A_n$. By using $g_n$ as a trial function, we get
\[
\lambda_p(A_n)\le \frac{\displaystyle\int_{A_n} |\nabla g_n|^p\,dx}{\displaystyle\int_{A_n} |g_n|^p\,dx}\le \frac{\displaystyle\int_{A_n} |\nabla g_n|^p\,dx}{\displaystyle\int_{E_n} |g_n|^p\,dx}\le \left(\frac{2}{\varepsilon}\right)^p\,\frac{C}{\delta}.
\] 
This implies that 
\[
r_n\ge \frac{\alpha_{N,p}}{\Big(\lambda_p(A_n)\Big)^\frac{1}{p}}\ge \frac{\varepsilon}{2}\,\left(\frac{\delta}{C}\right)^\frac{1}{p}\,\alpha_{N,p},\qquad \text{for every}\ n\in\mathbb{N}.
\]
We thus have proven \eqref{radius_uniform}.
\par
We now observe that the family of hyper-rectangles $\{D_{\mathbf{t}}(\mathbf{i}+\overline{Q_{1/2}})\}_{\mathbf{i}\in\mathbb{Z}^N}$ tile the whole space. Accordingly, for every $n\in\mathbb{N}$, there exists at least a point $\mathbf{i_n}\in\mathbb{Z}^N$ such that
\[
x_n\in D_{\mathbf{t}}(\mathbf{i}_n+\overline{Q_{1/2}}).
\]
Observe that by construction we have 
\[
B_{r_n}(x_n)\subseteq B_{R_n}(D_{\mathbf{t}}(\mathbf{i}_n)),\qquad \text{where}\ R_n=r_n+\frac{\sqrt{N}}{2}\,\max\{t_1,\dots,t_N\}.
\]
From \eqref{mezzino} and \eqref{radius_uniform}, we get the uniform lower bound
\[
|B_{R_n}(D_{\mathbf{t}}(\mathbf{i}_n))\cap A_n|\ge \frac{\omega_N\,c_1^N}{2},\qquad \text{for every}\ n\in\mathbb{N}.
\]
 Thanks to the uniform bounds \eqref{radius_uniform} on $\{r_n\}_{n\in\mathbb{N}}$, we can extract a subsequence (not relabelled) such that 
\[
\lim_{n\to\infty} r_n=\overline{r}>0.
\]
Thus, if we set $\overline{R}=\overline{r}+\max\{t_1,\dots,t_N\}\,\sqrt{N}/2$, we easily get
\begin{equation}
\label{forticara}
\lim_{n\to\infty} \|1_{B_{R_n}}-1_{B_{\overline R}}\|_{L^{p'}(\mathbb{R}^N)}=0.
\end{equation}
We now observe that the translated sequence $\{u_n(D_{\mathbf{t}}(\mathbf{i}_n)+\cdot)\}_{n\in\mathbb{N}}$ still belongs to $W^{1,p}_0(\Omega)$, thanks to the assumption on $\Omega$. Moreover, by recalling that $\lambda_{p}(\Omega)>0$, we have that the sequence  $\{u_n(D_{\mathbf{t}}(\mathbf{i}_n)+\cdot)\}_{n\in\mathbb{N}}$ is bounded in  $W^{1,p}(\Omega)$.
Hence, it weakly converges in $W^{1,p}(\Omega)$ to a function $u\in W^{1,p}_0(\Omega)$, up to a subsequence.
In particular, by testing the weak convergence in $L^p$ against the characteristic function of $B_{\overline{R}}$ and using \eqref{forticara}, we thus have
\[
\begin{split}
\int_{B_{\overline{R}}} u(x)\,dx=\lim_{n\to\infty} \int_{B_{\overline{R}}} u_n(D_{\mathbf{t}}(\mathbf{i}_n)+x)\,dx&=\lim_{n\to\infty}\int_{B_{R_n}} u_n(D_{\mathbf{t}}(\mathbf{i}_n)+x)\,dx\\
&=\lim_{n\to\infty} \int_{B_{R_n}(D_\mathbf{t}(\mathbf{i}_n))} u_n(y)\,dy\\
&\ge \liminf_{n\to\infty} \int_{B_{R_n}(D_{\mathbf{t}}(\mathbf{i}_n))\cap A_n} u_n(y)\,dy\\
&\ge \frac{\varepsilon}{4}\,\liminf_{n\to\infty}|B_{R_n}(D_{\mathbf{t}}(\mathbf{i}_n))\cap A_n|.
\end{split}
\]
In the last inequality we used that $u_n> \varepsilon/4$ on $A_n$, by construction.
In light of the previous uniform lower bound on the measure of $B_{R_n}(D_{\mathbf{t}}(\mathbf{i}_n))\cap A_n$, this is enough to infer that $u\not\equiv 0$.
\vskip.2cm\noindent
{\it Part 2: general case.} We now suppose that $\{u_n\}_{n\in\mathbb{N}}\subseteq W^{1,p}_0(\Omega)$. By definition, for every $n\in\mathbb{N}$ there exists a sequence of functions $\{u_{k,n}\}_{k\in\mathbb{N}}\subseteq C^\infty_0(\Omega)$ such that
\[
\lim_{k\to\infty} \|u_{k,n}-u_n\|_{L^p(\Omega)}=\lim_{k\to\infty}\|\nabla u_{k,n}-\nabla u_n\|_{L^p(\Omega)}=0.
\]
Since each $u_n$ is non-negative, we can construct $\{u_{k,n}\}_{n\in\mathbb{N}}$ with the same property.
Moreover, being $\{u_n\}_{n\in\mathbb{N}}$ bounded in $W^{1,p}_0(\Omega)$, we can further assume that
\[
\|u_{k,n}\|^p_{L^p(\Omega)}+\|\nabla u_{k,n}\|^p_{L^p(\Omega)}\le C,\qquad \text{for every}\ n,k\in\mathbb{N}.
\]
Thus, each sequence $\{u_{k,n}\}_{k\in\mathbb{N}}$ is bounded in $W^{1,p}(\Omega)$, uniformly in $n$. Thanks to the convergence in $L^p(\Omega)$, for every $n\in\mathbb{N}$ we choose $k_n\in\mathbb{N}$ such that
\[
\|u_{k,n}-u_n\|_{L^p(\Omega)}\le \frac{1}{n+1},\qquad \text{for every}\ k\ge k_n.
\]
Without loss of generality, we can choose $\{k_n\}_{n\in\mathbb{N}}$ to be increasing.
\par
We take $\varepsilon,\delta>0$ as in the statement. By the Markov-Chebyshev inequality, for every $n\in\mathbb{N}$ and $k\ge k_n$ we have
\[
\left|\Big\{x\in\Omega\, :\, |u_{k,n}(x)-u_n(x)|>\varepsilon/2\Big\}\right|\le \left(\frac{2}{\varepsilon}\right)^p\,\|u_{k,n}-u_n\|^p_{L^p(\Omega)}\le \left(\frac{2}{\varepsilon}\right)^p\,\frac{1}{(n+1)^p}.
\]
We choose $n_0=n_0(p,\varepsilon,\delta)>0$ such that
\[
\left(\frac{2}{\varepsilon}\right)^p\,\frac{1}{(n_0+1)^p}<\frac{\delta}{2}.
\]
Then, for every $n\ge n_0$ we have
\[
\left|\Big\{x\in\Omega\, :\, |u_{k,n}(x)-u_n(x)|>\varepsilon/2\Big\}\right|<\frac{\delta}{2},\qquad \text{for every}\ k\ge k_n,
\]
Observe that 
\[
\Big\{x\in\Omega\, :\, u_n(x)>\varepsilon\Big\}\subseteq \Big\{x\in\Omega\, :\, |u_{k,n}(x)-u_n(x)|>\varepsilon/2\Big\}\cup \Big\{x\in\Omega\, :\, u_{k,n}(x)>\varepsilon/2\Big\}.
\]
Accordingly, we get that for every $n\ge n_0$ the function $U_n:=u_{k_n,n}$ is such that
\[
\left|\Big\{x\in\Omega\, :\, U_n(x)>\varepsilon/2\Big\}\right|\ge \left|\Big\{x\in\Omega\, :\, u_n(x)>\varepsilon\Big\}\right|-\left|\Big\{x\in\Omega\, :\, |U_n(x)-u_n(x)|>\varepsilon/2\Big\}\right|\ge \frac{\delta}{2}.
\]
We can thus apply the first part of the proof to the sequence $\{U_n\}_{n\in \mathbb{N}}$ and we can infer that
there exists a sequence of vectors $\{\mathbf{i}_n\}_{n\in\mathbb{N}}\subseteq \mathbb{Z}^N$ and  a function $u\in W^{1,p}_0(\Omega)\setminus\{0\}$ such that 
\[
\lim_{n\to\infty} \int_\Omega (u-U_n(D_{\mathbf{t}}(\mathbf{i}_n)+\cdot))\,\varphi\,dx=0 ,\qquad \text{for every}\ \varphi\in L^{p'}(\Omega),
\]
up to a subsequence. In particular, thanks to the periodicity of $\Omega$, we also get
\[
\begin{split}
\lim_{n\to\infty} \left|\int_\Omega (u-u_n(D_{\mathbf{t}}(\mathbf{i}_n)+\cdot))\,\varphi\,dx\right|&\le \lim_{n\to\infty} \left|\int_\Omega (u-U_n(D_{\mathbf{t}}(\mathbf{i}_n)+\cdot)))\,\varphi\,dx\right|\\
&+\lim_{n\to\infty} \|U_n-u_n\|_{L^p(\Omega)}\,\|\varphi\|_{L^{p'}(\Omega)}=0.
\end{split}
\]
Taking into account that the sequence $\{u_n (D_{\mathbf{t}}(\mathbf{i}_n)+\cdot) \}_{n\in\mathbb{N}}\subseteq W^{1,p}_0(\Omega)$  is bounded, the last inequality implies that  $\{u_n (D_{\mathbf{t}}(\mathbf{i}_n)+\cdot)\}_{n\in\mathbb{N}}$
converges weakly in $W^{1,p}_0(\Omega)$ to the  function $u\not\equiv0$, as desired.
\end{proof}

\section{A well-prepared minimizing sequence}
\label{sec:5}

As in \cite{BraBriPri}, we will construct a suitable minimizing sequence for our original problem by adding a ``vanishing confinement'' term. In other words, we will use the following simple result, which is the same as \cite[Lemma 4.1]{BraBriPri}. We use the following notation
\[
W^{1,p}_0(\Omega;|x|):=\left\{u\in W^{1,p}_0(\Omega)\, :\, \int_\Omega |x|\,|u|^p\,dx<+\infty\right\}.
\] 
\begin{lm}
\label{lm:vanconf}
Let $1<p<\infty$ and $q\ge p$ satisfying \eqref{exponents}. Let $\Omega\subseteq\mathbb{R}^N$ be an open set such that $\lambda_p(\Omega)>0$.
For every $n\in\mathbb{N}$, we define 
\[
\lambda_{p,q}(\Omega;V_n):=\inf_{u\in W^{1,p}_0(\Omega;|x|)}\Big\{\mathcal{G}_{p,n}(u)\,:\, \|u\|_{L^q(\Omega)}=1\Big\},
\]
where 
\[
\mathcal{G}_{p,n}(u)=\int_{\Omega}|\nabla u|^{p}\,dx +\int_{\Omega}V_n\,|u|^{p}\,dx\qquad \text{and}\qquad V_n(x)=\frac{|x|}{n+1}.
\]
Then, we have
\begin{equation}
\label{convauto}
\lim_{n\to\infty}\lambda_{p,q}(\Omega;V_n)=\lambda_{p,q}(\Omega).
\end{equation}
Moreover, the value $\lambda_{p,q}(\Omega;V_n)$ is attained by some non-negative function $u_n\in W^{1,p}_0(\Omega;|x|)$ with $\|u_n\|_{L^q(\Omega)}=1$ and we have
\begin{equation}
\label{minsequence}
\lim_{n\to\infty}\int_{\Omega}|\nabla u_n|^p\, dx=\lambda_{p,q}(\Omega).
\end{equation}
\end{lm}
\begin{rem}
The fact that $u_n$ can be chosen to be non-negative is standard: indeed, $|u_n|$ is still admissible,  while both the functional and the constraint are invariant by the change $u\mapsto|u|$. Moreover, by comparing \eqref{convauto} and \eqref{minsequence}, we easily get the following useful information
\begin{equation}
\label{azzero}
\lim_{n\to\infty} \int_\Omega V_n\,|u_n|^p\,dx=0.
\end{equation}
Finally, by minimality, we remark that each $u_n$ is a weak solution of the following Euler-Lagrange equation
\begin{equation}
\label{eqvan}
-\Delta_p u_n+V_n\,u_n^{p-1}=\lambda_{p,q}(\Omega;V_n)\,u_n^{q-1},\qquad \text{in}\ \Omega.
\end{equation}
\end{rem}
\begin{rem}[Uniform $L^\infty$ bound]
By using that the sequence  $n\mapsto \lambda_{p,q}(\Omega;V_n)$ is non-increasing, we get that each $u_n$ is a weak subsolution of
\[
-\Delta_p u_n\le \lambda_{p,q}(\Omega;V_n)\,u_n^{q-1}\leq  \lambda_{p,q}(\Omega;V_1)\,u_n^{q-1}\ ,\qquad \text{in}\ \Omega.
\]
Thus, by appealing to \cite[Lemma 2.3]{BraBriPri}, we get that $u_n\in L^\infty(\Omega)$ together with the following uniform estimate
\begin{equation}
\label{stimainfinito}
\|u_n\|_{L^\infty(\Omega)}\le C_{N,p,q}\,\Big(\lambda_{p,q}(\Omega;V_1)\Big)^\frac{N}{p\,q-(q-p)\,N}=:M.
\end{equation}
Observe that we also used that $u_n$ has unit $L^q(\Omega)$ norm. 
\end{rem}
Actually, the property \eqref{azzero} can be improved. This is the content of the next simple result, which will be crucially exploited. The proof is a bit lenghty, though elementary.
\begin{lm}
\label{lm:fortissimo}
With the notation of the previous result, we have 
\[
\lim_{n\to\infty}\int_\Omega V_n^k\,|\nabla u_n|^p\,dx=\lim_{n\to\infty} \int_\Omega V_n^{k}\,|u_n|^p\,dx=0,\qquad \text{for every}\ k\in\mathbb{N}\setminus\{0\}.
\]
\end{lm}
\begin{proof}
We fix $R>0$ and take $\eta\in C^\infty_0(B_{R+1})$ a cut-off function such that 
\[
0\le \eta\le 1,\qquad\eta\equiv 1\ \text{on}\ B_R,\qquad \|\nabla \eta\|_{L^\infty}\le C,
\]
for a universal constant $C>0$. We test the weak formulation of \eqref{eqvan} with the choice
\[
\varphi=V_n^k\,u_n\,\eta^p.
\]
We then obtain
\begin{equation}
\label{daldo}
\begin{split}
\int_\Omega |\nabla u_n|^p\,V_n^k\,\eta^p\,dx&+p\,\int_\Omega \langle |\nabla u_n|^{p-2}\,\nabla u_n,\nabla \eta\rangle\,\eta^{p-1}\,u_n\,V_n^k\,dx\\
&+\int_\Omega \langle |\nabla u_n|^{p-2}\,\nabla u_n,\nabla V_n^k\rangle\,u_n\,\eta^p\,dx\\
&+\int_\Omega V_n^{k+1}\,u_n^p\,\eta^p\,dx=\lambda_{p,q}(\Omega;V_n)\,\int_\Omega u_n^q\,V_n^k\,\eta^p\,dx.
\end{split}
\end{equation}
By using Cauchy-Schwarz and Young inequalities, we can estimate 
\[
\begin{split}
p\,\int_\Omega \langle |\nabla u_n|^{p-2}\,\nabla u_n,\nabla \eta\rangle\,\eta^{p-1}\,u_n\,V_n^k\,dx&\ge -p\,\int_\Omega  |\nabla u_n|^{p-1}\,|\nabla \eta|\,\eta^{p-1}\,u_n\,V_n^k\,dx\\
&\ge-\delta\,(p-1)\,\int_\Omega  |\nabla u_n|^{p}\,\eta^{p}\,V_n^k\,dx\\
&-\delta^{1-p}\,\int_\Omega V_n^k\,u_n^p\,|\nabla\eta|^p\,dx,
\end{split}
\]
for every $\delta>0$. Thus, from \eqref{daldo} we get
\begin{equation}
\label{daldo2}
\begin{split}
(1-(p-1)\,\delta)\,\int_\Omega |\nabla u_n|^p\,V_n^k\,\eta^p\,dx&+\int_\Omega \langle |\nabla u_n|^{p-2}\,\nabla u_n,\nabla V_n^k\rangle\,u_n\,\eta^p\,dx\\
&+\int_\Omega V_n^{k+1}\,u_n^p\,\eta^p\,dx\le \lambda_{p,q}(\Omega;V_n)\,\int_\Omega u_n^q\,V_n^k\,\eta^p\,dx\\
&+\delta^{1-p}\,\int_\Omega V_n^k\,u_n^p\,|\nabla\eta|^p\,dx.
\end{split}
\end{equation}
Analogously, by taking into account that $|\nabla V_n|=1/(n+1)$, we have
\[
\begin{split}
\int_\Omega \langle |\nabla u_n|^{p-2}\,\nabla u_n,\nabla V_n^k\rangle\,u_n\,\eta^p\,dx&\ge -\frac{k}{n+1}\,\int_\Omega |\nabla u_n|^{p-1}\,u_n\,V_n^{k-1}\,\eta^p\,dx\\
&\ge -\frac{k}{n+1}\,\frac{p-1}{p}\,\int_\Omega |\nabla u_n|^p\,V_n^{k-1}\,\eta^p\,dx\\
&-\frac{k}{n+1}\,\frac{1}{p}\,\int_\Omega u_n^p\,V_n^{k-1}\,\eta^p\,dx.
\end{split}
\]
By using this estimate in \eqref{daldo2}, we obtain
\[
\begin{split}
(1-(p-1)\,\delta)\,\int_\Omega |\nabla u_n|^p\,V^k_n\,\eta^p\,dx&+\int_\Omega V_n^{k+1}\,u_n^p\,\eta^p\,dx\\
&\le \lambda_{p,q}(\Omega;V_n)\,\int_\Omega u_n^q\,V_n^k\,\eta^p\,dx\\
&+\delta^{1-p}\,\int_\Omega V_n^k\,u_n^p\,|\nabla\eta|^p\,dx\\
&+\frac{k}{n+1}\,\frac{p-1}{p}\,\int_\Omega |\nabla u_n|^p\,V_n^{k-1}\,\eta^p\,dx\\
&+\frac{k}{n+1}\,\frac{1}{p}\,\int_\Omega u_n^p\,V_n^{k-1}\,\eta^p\,dx.
\end{split}
\]
This is valid for every $\delta>0$: we choose $\delta=1/(2\,(p-1))$ and use the properties of the cut-off function $\eta$, so to obtain
\[
\begin{split}
\frac{1}{2}\,\int_{\Omega\cap B_R} |\nabla u_n|^p\,V_n^k\,dx+\int_{\Omega\cap B_R} V_n^{k+1}\,u_n^p\,dx&\le \lambda_{p,q}(\Omega;V_n)\,\int_{\Omega\cap B_{R+1}} u_n^q\,V_n^k\,dx\\
&+C\,(2\,(p-1))^{p-1}\,\int_{\Omega\cap B_{R+1}} V_n^k\,u_n^p\,dx\\
&+\frac{k}{n+1}\,\frac{p-1}{p}\,\int_{\Omega\cap B_{R+1}} |\nabla u_n|^p\,V_n^{k-1}\,dx\\
&+\frac{k}{n+1}\,\frac{1}{p}\,\int_{\Omega\cap B_{R+1}} V_n^{k-1}\,u_n^p\,dx.
\end{split}
\]
On the right-hand side, we can further use that 
\[
\lambda_{p,q}(\Omega;V_n)\le \lambda_{p,q}(\Omega;V_1),
\]
and
\[
\int_{\Omega\cap B_{R+1}} u_n^q\,V^k_n\,dx\le \|u_n\|_{L^\infty(\Omega)}^{q-p}\,\int_{\Omega\cap B_{R+1}}  u_n^p\,V^k_n\,dx\le M^{q-p}\,\int_{\Omega\cap B_{R+1}}  u_n^p\,V^k_n\,dx,
\]
also thanks to \eqref{stimainfinito}. We also observe that
\[
\frac{1}{n+1}\,V_n^{k-1}=\frac{|x|^{k-1}}{(n+1)^k}\le \frac{|x|^k}{(n+1)^k}+\frac{1}{(n+1)^k}=V_n^k+\frac{1}{(n+1)^k}.
\]
With simple manipulations, we can thus obtain
\[
\begin{split}
\int_{\Omega\cap B_R} |\nabla u_n|^p\,V_n^k\,dx+\int_{\Omega\cap B_R} V_n^{k+1}\,u_n^p\,dx&\le \left(C+\frac{2}{p}\, k\right)\,\int_{\Omega\cap B_{R+1}} u_n^p\,V_n^k\,dx\\
&+\frac{2\,k}{n+1}\,\frac{(p-1)}{p}\,\int_{\Omega\cap B_{R+1}} |\nabla u_n|^p\,V_n^{k-1}\,dx\\
&+\frac{2\,k}{(n+1)^k}\,\frac{1}{p}\,\int_{\Omega\cap B_{R+1}} u_n^p\,dx,
\end{split}
\]
for a constant $C=C(N,p,q,\Omega)>0$. Finally, we estimate the last $L^p$ norm as follows 
\[
\int_{\Omega\cap B_{R+1}} u_n^p\,dx\le \int_\Omega u_n^p\,dx\le \frac{1}{\lambda_p(\Omega)}\,\int_\Omega |\nabla u_n|^p\,dx\le \frac{\lambda_{p,q}(\Omega;V_n)}{\lambda_p(\Omega)}\le \frac{\lambda_{p,q}(\Omega;V_1)}{\lambda_p(\Omega)}.
\]
If we set for simplicity
\[
\mathcal{I}_{n,k}(R)=\int_{\Omega\cap B_R} u_n^p\,V_n^k\,dx\quad \text{and}\quad \mathcal{J}_{n,k}(R)=\int_{\Omega\cap B_R} |\nabla u_n|^p\,V_n^{k-1}\,dx,\qquad n\in\mathbb{N},\ k\in\mathbb{N}\setminus\{0\},
\]
we have finally obtained the interlaced recursive estimate
\begin{equation}
\label{recorsa}
\mathcal{I}_{n,k+1}(R)+\mathcal{J}_{n,k+1}(R)\le \left(C+\frac{2}{p}\,  k\right)\,\mathcal{I}_{n,k}(R+1)+\frac{C\, k}{n+1}\,\mathcal{J}_{n,k}(R+1)+\frac{C\,k}{(n+1)^k},
\end{equation}
possibly for a different constant $C=C(N,p,q,\Omega)>0$. By proceeding inductively on $k\in\mathbb{N}\setminus\{0\}$, we can obtain at first that
\[
\mathcal{I}_{n,k}:=\int_\Omega V_n^{k}\,u_n^p\,dx<+\infty\qquad \text{and}\qquad \mathcal{J}_{n,k}:=\int_\Omega V_n^{k-1}\, |\nabla u_n|^p\,dx<+\infty.
\]
Indeed, this is true for $k=1$ by construction. Moreover, if $\mathcal{I}_{n,k}<+\infty$ and $\mathcal{J}_{n,k}<+\infty$ for a certain $k\ge 1$, by taking the limit as $R$ goes to $+\infty$ in \eqref{recorsa} and using the Monotone Convergence Theorem, we get that $\mathcal{I}_{n,k+1}<+\infty$ and $\mathcal{J}_{n,k+1}<+\infty$, as well. In particular, from \eqref{recorsa} we get 
\begin{equation}
\label{rincorsa}
\mathcal{I}_{n,k+1}+\mathcal{J}_{n,k+1}\le \left(C+\frac{2}{p}\,k\right)\,\mathcal{I}_{n,k}+\frac{k\, C}{n+1}\,\mathcal{J}_{n,k}+\frac{C\,k}{(n+1)^k},
\end{equation}
for a constant $C=C(N,p,q,\Omega)>0$. The claim can now be easily obtained from \eqref{rincorsa} by an induction argument over $k$. It is sufficient to observe that for $k=1$ we have
\[
\lim_{n\to\infty}\mathcal{I}_{n,1}=\lim_{n\to\infty} \int_\Omega V_n\,u_n^p\,dx=0,
\]
by \eqref{azzero}, while obviously
\[
\lim_{n\to\infty} \frac{1}{n+1}\,\mathcal{J}_{n,1}=\lim_{n\to\infty} \frac{1}{n+1}\,\int_\Omega |\nabla u_n|^p\,dx=0,
\]
thanks to \eqref{minsequence}. This concludes the proof.
\end{proof}

\section{Existence of extremals: periodic sets}
\label{sec:6}
We are now ready for the main result of the paper.
\begin{teo}[Case $q<\infty$]
\label{teo:main}
Let $1<p<\infty$ and assume that  $q>p$ satisfies \eqref{exponents}, with $q<\infty$. Let $\mathbf{t}=(t_1,\dots,t_N)\in\mathbb{R}^N$ be such that $t_i>0$, for every $i\in\{1,\dots,N\}$ and let $K\subsetneq \overline{Q_{1/2}}$ be a compact set such that
\[
\mathrm{cap}_p(K; Q_1)>0.
\]
Let $\Omega\subseteq\mathbb{R}^N$ be a $\mathbf{t}-$periodically perforated open set, generated by $K$ (recall Definition \ref{defi:apps_intro}).
Then, the infimum defining $\lambda_{p,q}(\Omega)$ is attained by some non-negative function $u\in W^{1,p}_0(\Omega)\setminus\{0\}$, such that
\[
u\in L^\infty(\Omega)\qquad \text{and}\qquad \|u\|_{L^\infty(D_{\mathbf{t}}(Q_{1/2}\setminus K))}=\|u\|_{L^\infty(\Omega)}.
\]
\end{teo}
\begin{proof}
We first observe that $\lambda_p(\Omega)>0$, thanks to Lemma \ref{lm:nonbanale}. By \eqref{mara}, this in turn implies that $\lambda_{p,q}(\Omega)>0$, as well.
\par
We take $\{u_n\}_{n\in\mathbb{N}}$ to be the minimizing sequence for $\lambda_{p,q}(\Omega)$ constructed in Lemma \ref{lm:vanconf}. 
There is no loss of generality in assuming that
\begin{equation}\label{gradientilimitati}
\|\nabla u_n\|^p_{L^p(\Omega)}\le2\,\lambda_{p,q}(\Omega),\qquad \text{for every}\ n\in\mathbb{N}.
\end{equation} 
Moreover, we notice that it holds
\begin{equation}\label{marta}
\|u_n\|_{L^q(\Omega)}=1,\qquad \|u_n\|_{L^p(\Omega)}\le \left(\frac{2\,\lambda_{p,q}(\Omega)}{\lambda_p(\Omega)}
\right)^{\frac {1}{p}},
\end{equation}
and, by the Gagliardo-Nirenberg interpolation inequality, we have
\begin{equation}\label{r}
\|u_n\|_{L^r(\Omega)}\le G_{N,p,r} \|\nabla u_n\|^\theta_{L^p(\Omega)}\,\| u_n\|^{1-\theta}_{L^p(\Omega)}\le \left(2\,\lambda_{p,q}(\Omega)\right)^{\frac{1}{p}}\,\Big(\lambda_p(\Omega)\Big)^\frac{\theta-1}{p}.
\end{equation}
for every $q<r<p^*$ when $p<N$, and for every $q<r<\infty$ when $p\ge N$.
Here $\theta=\theta(N,p,r)\in(0,1)$ is an exponent dictated by scale invariance, its precise value has no bearing. 
\par
In particular, \eqref{marta} and \eqref{r} ensure the applicability  of the so-called {\it $pqr-$Lemma}, originally devised in \cite[Lemma 2.1]{FLL}, to the sequence $\{u_n\}_{n\in\mathbb{N}}$. Accordingly, we can find uniform constants $\varepsilon,\delta>0$ such that
\begin{equation}\label{pqr}
\left|\Big\{x\in\Omega\, :\, u_n(x)>\varepsilon\Big\}\right|>\delta,\quad \text{ for all }n\in\mathbb{N}.
\end{equation}
Thanks to \eqref{gradientilimitati} and \eqref{pqr} we can apply Lemma \ref{lm:lieb}: let  $\{\mathbf{i}_n\}_{n\in\mathbb{N}}\subseteq\mathbb{Z}^N$  be the corresponding sequence of vectors. In order to simplify the notation, we set 
\[
\widetilde{u}_n(x):=u_n(D_{\mathbf{t}}(\mathbf{i}_n)+x).
\] 
Thus, the sequence $\{\widetilde{u}_n\}_{n\in\mathbb{N}}$ weakly converges in $W^{1,p}(\Omega)$ to some $\widetilde{u}\in W^{1,p}_0(\Omega)\setminus\{0\}$. Thanks to the periodicity of $\Omega$, we have that $\{\widetilde{u}_n\}_{n\in\mathbb{N}}$ is still a minimizing sequence for $\lambda_{p,q}(\Omega)$. Indeed, we have 
\[
\|\widetilde{u}_n\|_{L^q(\Omega)}=\|u_n\|_{L^q(\Omega)}=1,\qquad \lim_{n\to\infty}\int_\Omega |\nabla \widetilde{u}_n|\,dx=\lim_{n\to\infty}\int_\Omega |\nabla u_n|\,dx=\lambda_{p,q}(\Omega).
\]
By virtue of the weak lower semi-continuity of the $L^p$ norm with respect to the weak convergence, in order to prove the theorem we only need to show that 
\begin{equation}\label{normalimite}
\|\widetilde{u}\|_{L^q(\Omega)}=1.
\end{equation}
To this aim, the key step consists in proving that both the sequence of functions $\{\widetilde{u}_n\}_{n\in\mathbb{N}}$ and the sequence of gradients $\{\nabla\widetilde{ u}_n\}_{n\in\mathbb{N}}$ converge almost everywhere in $\Omega$. With this property at our disposal, we will get \eqref{normalimite} by applying the same argument as in \cite[Lemma 2.7]{LiHLS}. 
\par
The almost everywhere convergence for the sequence of functions is easily inferred. Indeed, we can consider $\{\widetilde{u}_n\}_{n\in\mathbb{N}}$ and $\widetilde{u}$ as defined on the whole $\mathbb{R}^N$, by extending them to be $0$ on the complement of $\Omega$. The extended functions belong to $W^{1,p}(\mathbb{R}^N)$ (see for example \cite[Lemma 3.7.9]{BraBook}) and we still have weak convergence in this space.
In particular, we have $\{\widetilde{u}_n\}_{n\in\mathbb{N}}\subseteq W^{1,p}(B_R)$, for every $R>0$.
Hence, by exploiting the compactness of the embedding $W^{1,p}(B_R)\hookrightarrow L^p(B_R)$, the weak convergence in $W^{1,p}_0(\Omega)$ exposed above and  a standard argument (see for example \cite[Lemma 3.8.7 \& Remark 3.9.5]{BraBook}) we obtain 
\begin{equation}
\label{convpalle}
\lim_{n\to\infty} \|\widetilde{u}_n-\widetilde{u}\|_{L^p(B_R)}=\lim_{n\to\infty} \|\widetilde{u}_n-\widetilde{u}\|_{L^p(\Omega\cap B_R)}=0,\qquad \text{for every}\ R>0.
\end{equation} 
In order to treat the sequence of gradients, we first observe that, thanks to Lemma \ref{lm:fortissimo} and using that
\[
\Omega-D_\mathbf{t}(\mathbf{i}_n)=\Omega,\qquad \text{for every}\ n\in\mathbb{N} ,
\]
we have
\begin{equation}
\label{fortissima_traslata}
\lim_{n\to\infty} \int_\Omega \widetilde{V}_n^{k+1}\,|\widetilde{u}_n|^p\,dx=0,\qquad \text{where}\ \widetilde{V}_n(x):=V_n(x+D_\mathbf{t}(\mathbf{i}_n)),
\end{equation}
for every $k\in\mathbb{N}$. Furthermore, it is elementary to verify that each $\widetilde{u}_n$ solves the following problem\footnote{Observe that for every given $n\in\mathbb{N}$, we have
\[
\int_\Omega \widetilde{V}_n\,|u|^p\,dx=\frac{1}{n+1}\,\int_\Omega |x+D_\mathbf{t}(\mathbf{i}_n)|\,|u|^p\,dx<+\infty,\qquad \text{for every}\ u\in W^{1,p}_0(\Omega;|x|).
\]
It is sufficient to observe that 
\[
\int_\Omega |x+D_\mathbf{t}(\mathbf{i}_n)|\,|u|^p\,dx\le \int_\Omega |x|\,|u|^p\,dx+|D_\mathbf{t}(\mathbf{i}_n)|\,\int_\Omega |u|^p\,dx,
\]
and use that both terms are finite, thanks to the assumption $u\in W^{1,p}_0(\Omega;|x|)$.} 
\[
\inf_{\varphi\in W^{1,p}_0(\Omega;|x|)}\left\{\int_{\Omega}|\nabla \varphi|^{p}\,dx +\int_{\Omega}\widetilde{V}_n\,|\varphi|^{p}\,dx\,:\, \|\varphi\|_{L^q(\Omega)}=1\right\},
\]
 and that this infimum still coincides with $\lambda_{p,q}(\Omega;V_n)$.
Thus, $\widetilde{u}_n$ weakly solves the relevant Euler-Lagrange equation
\begin{equation}
\label{eqtrans}
-\Delta_p \widetilde{u}_n+\widetilde{V}_n\,\widetilde u_n^{p-1}=\lambda_{p,q}(\Omega;V_n)\,\widetilde u_n^{q-1},\qquad \text{in}\ \Omega.
\end{equation}
We now proceed as in the proof of \cite[Theorem 5.1]{BraBriPri}, but this time we have to pay attention to the presence of the translations $x\mapsto x+D_\mathbf{t}(\mathbf{i}_n)$.
Let $R>0$ and $\eta\in C^{\infty}_0(\mathbb{R}^N)$ be a cut-off function such that
\[
0\le \eta\le 1,\qquad \eta=1 \text{ on } B_{R},\qquad \eta=0\ \text{on}\ \mathbb{R}^{N}\setminus B_{2R}, \qquad \|\nabla \eta\|_{\infty}\le \frac{C}{R}.
\]
We insert in the weak formulation of \eqref{eqtrans} the admissible test function $\eta\, (\widetilde u_n-\widetilde{u})\in W^{1,p}_0(\Omega)$. This yields
\[
\begin{split}
\int_\Omega \langle|\nabla \widetilde u_n|^{p-2}\,\nabla \widetilde u_n, \nabla \widetilde u_n-\nabla \widetilde{u}\rangle\,\eta\,dx &=\lambda_{p,q}(\Omega; V_n)\,\int_{\Omega}\widetilde u_n^{q-1}\,(\widetilde u_{n}-\widetilde{u})\,\eta\,dx\\
&-\int_{\Omega}\widetilde{V}_n\,\widetilde u_n^{p-1}\,(\widetilde u_n-\widetilde{u})\,\eta\,dx\\
&-\int_\Omega \langle|\nabla \widetilde u_n|^{p-2}\,\nabla \widetilde u_n, \nabla \eta\rangle\,(\widetilde u_n-\widetilde{u}) \,dx .
\end{split}
\]
The first and third integrals in the right-hand side converge to $0$, as $n$ goes to $\infty$: this fact is quite straightforward. Indeed, by exploiting the uniform upper bounds \eqref{stimainfinito},   \eqref{gradientilimitati} and the properties of the cuf-off function $\eta$, we have
\[
\lambda_{p,q}(\Omega; V_n)\,\left|\int_{\Omega}\widetilde u_n^{q-1}\,(\widetilde u_{n}-\widetilde{u})\,\eta\,dx\right|\le \lambda_{p,q}(\Omega;V_n)\,M^{q-1}\,|B_{2R}|^\frac{p-1}{p}\,\|\widetilde u_n-\widetilde{u}\|_{L^p(B_{2R})},
\]
and
\[
\left|\int_\Omega \langle|\nabla \widetilde u_n|^{p-2}\,\nabla \widetilde u_n, \nabla \eta\rangle\,(\widetilde u_n-\widetilde{u}) \,dx\right|\le \frac{C}{R}\,\left(2\,\lambda_{p,q}(\Omega)\right)^{\frac{p-1}{p}}\,\|\widetilde u_n-\widetilde{u}\|_{L^p(B_{2R})}.
\]
If we use \eqref{convpalle}, we get the claim. On the contrary, the second integral
\[
\left|\int_{\Omega}\widetilde{V}_n\,\widetilde u_n^{p-1}\,(\widetilde u_n-\widetilde{u})\,\eta\,dx\right|,
\]
is more delicate, due to the presence of the translation vectors $\{\mathbf{i}_n\}_{n\in\mathbb{N}}$, for whose norms we do not have any estimate. However, we can simply apply H\"older's inequality 
and get
\[
\left|\int_{\Omega}\widetilde{V}_n\,\widetilde u_n^{p-1}\,(\widetilde u_n-\widetilde{u})\,\eta\,dx\right|\le \left(\int_\Omega \widetilde{V}_n^\frac{p}{p-1}\,\widetilde{u}_n^p\,dx\right)^\frac{p-1}{p}\,\|\widetilde u_n-\widetilde{u}\|_{L^p(B_{2R})}.
\]
Observe that the first term on the right-hand side is uniformly bounded, thanks to \eqref{fortissima_traslata}. Thus, this term converges to $0$, as well. We can finally take the limit as $n$ goes to $\infty$, so to infer 
\[
\lim_{n\to\infty}\int_{\Omega} \langle|\nabla \widetilde u_n|^{p-2}\,\nabla \widetilde u_n, \nabla \widetilde u_n-\nabla \widetilde{u}\rangle\,\eta\,dx=0.
\]
On the other hand, by testing the previously inferred weak convergence of $\nabla \widetilde{u}_n$ to $\nabla \widetilde{u}$ against $|\nabla \widetilde{u}|^{p-2}\,\nabla \widetilde{u}\,\eta\in L^{p'}(\Omega)$, we obtain
\[
\lim_{n\to\infty}\int_{\Omega} \langle|\nabla \widetilde{u}|^{p-2}\,\nabla \widetilde{u}, \nabla \widetilde u_n-\nabla \widetilde{u}\rangle\,\eta\,dx=0.
\]
By subtracting the last two identities, we deduce that
\[
\lim_{n\to\infty}\int_{\Omega} \langle|\nabla \widetilde u_n|^{p-2}\,\nabla \widetilde u_n-|\nabla \widetilde{u}|^{p-2}\,\nabla \widetilde{u}, \nabla \widetilde u_n-\nabla \widetilde{u}\rangle\,\eta\,dx=0.
\]
By exploiting the monotonicity properties of the $p-$Laplacian (see for example the proof of \cite[Lemma B.1]{BPZ}), this information is enough to obtain that
\begin{equation}
\label{convpallegrad}
\lim_{n\to\infty}\|\nabla \widetilde u_n-\nabla \widetilde{u}\|_{L^{p}(B_R)}=\lim_{n\to\infty}\|\nabla \widetilde u_n-\nabla \widetilde{u}\|_{L^{p}(\Omega\cap B_R)}=0,\qquad \text{for every}\ R>0,
\end{equation}
as desired.
\par
We are now in a position to repeat {\it verbatim}  the last steps of the proof of \cite[Theorem 5.1]{BraBriPri}. Let us briefly recall them, for the reader's convenience. 
 The inferred convergences \eqref{convpalle} and \eqref{convpallegrad} allow to conclude that there exist a sub-sequence (not relabeled) $\{\widetilde{u}_n\}_{n\in\mathbb{N}}$ that converges to $\widetilde{u}$ almost everywhere in $\Omega$ and such that also the sequence $\{\nabla\widetilde{u}_n\}_{n\in\mathbb{N}}$  converges to $\nabla u$ almost everywhere in $\Omega$. Hence, the {\it Brezis-Lieb Lemma} (see \cite[Theorem 1]{BreLie}) applies to both the sequence of functions and that of gradients. We get 
\begin{equation}\label{eq:breliefun}
\lim_{n\to\infty} A_n:=\lim_{n\to\infty}\left(\int_\Omega |\widetilde{u}_n|^{q}\,dx-\int_{\Omega}| \widetilde{u}_n- \widetilde{u}|^{q}\,dx\right)=\int_{\Omega}|\widetilde{u}|^{q}\,dx,
\end{equation}
and
\begin{equation}
 \label{steiner1}
\lim_{n\to\infty}\int_{\Omega}|\nabla \widetilde{u}_n-\nabla \widetilde{u}|^{p}\,dx=\lambda_{p,q}(\Omega)-\int_{\Omega}|\nabla \widetilde{u}|^{p}\,dx,
\end{equation}
respectively, where in the second identity we used also that $\{\widetilde{u}_n\}_{n\in\mathbb{N}}$ is a minimizing sequence for $\lambda_{p,q}(\Omega)$. 
Being $\widetilde{u}\not\equiv 0$, we deduce from \eqref{eq:breliefun} that there exists $n_0\in\mathbb{N}$, such that $A_n$ is uniformly bounded from below by a positive constant, for every $n>n_0$. 
Then, a simple quantified version of the sub-additivity of the power function $t\mapsto t^{p/q}$ (see for example \cite[Lemma 2.1]{BraBriPri}) and the definition of $\lambda_{p,q}(\Omega)$ give us
\[
\begin{split}
\lambda_{p,q}(\Omega)&=\lambda_{p,q}(\Omega)\,\|\widetilde{u}_n\|_{L^q(\Omega)}^p\\
&\le \lambda_{p,q}(\Omega)\,A_n^\frac{p}{q}+\lambda_{p,q}(\Omega)\,\|\widetilde{u}_n-\widetilde{u}\|_{L^{q}(\Omega)}^p-c\,\lambda_{p,q}(\Omega)\,\min\left\{A_n^\frac{p}{q},\|\widetilde{u}_n-\widetilde{u}\|_{L^{q}(\Omega)}^p\right\}\\
&\le \lambda_{p,q}(\Omega)\,A_n^\frac{p}{q}+\int_{\Omega}|\nabla \widetilde{u}_n-\nabla \widetilde{u}|^{p}\,dx-c\,\lambda_{p,q}(\Omega)\,\min\left\{A_n^\frac{p}{q},\|\widetilde{u}_n-\widetilde{u}\|_{L^{q}(\Omega)}^p\right\},
\end{split}
\]
for some $c=c(p/q)>0$.
By taking the limit as $n$ goes to $\infty$ and applying  \eqref{eq:breliefun} and \eqref{steiner1}, we obtain
\[
\int_\Omega |\nabla \widetilde{u}|^p\,dx+ c\,\lambda_{p,q}(\Omega)\,\lim_{n\to\infty}\min\left\{A_n^\frac{p}{q},\|\widetilde{u}_n-\widetilde{u}\|_{L^{q}(\Omega)}^p\right\}\le \lambda_{p,q}(\Omega)\,\left(\int_\Omega |\widetilde{u}|^q\,dx\right)^\frac{p}{q}.
\]
On the other hand, by definition of $\lambda_{p,q}(\Omega)$, we have
\[
\int_\Omega |\nabla \widetilde{u}|^p\,dx\ge \lambda_{p,q}(\Omega)\,\left(\int_\Omega |\widetilde{u}|^q\,dx\right)^\frac{p}{q}.
\]
By joining the last two inequalities, we get in particular that
\[
\lim_{n\to\infty}\min\left\{A_n^\frac{p}{q},\|\widetilde{u}_n-\widetilde{u}\|_{L^{q}(\Omega)}^p\right\}=0.
\]
By recalling that $A_n$ is uniformly bounded from below by a positive constant, the latter implies that $\{\widetilde{u}_n\}_{n\in\mathbb{N}}$ converges strongly in $L^q(\Omega)$ to $\widetilde{u}$ and hence \eqref{normalimite}.
This concludes the proof of the existence.
\par
The fact that $\widetilde{u}\in L^\infty(\Omega)$ follows from \cite[Lemma 2.3]{BraBriPri}. We now claim that, up to translating $\widetilde{u}$, we can guarantee that 
\[
\|\widetilde{u}\|_{L^\infty(D_{\mathbf{t}}(Q_{1/2}\setminus K))}=\|\widetilde{u}\|_{L^\infty(\Omega)}.
\]
Indeed, by \cite[Theorem 7.3]{BraBriPri}, we know that $\widetilde{u}$ (exponentially) decays to $0$, at infinity. This shows that there exists $k\in\mathbb{N}\setminus\{0\}$ such that
\[
\|\widetilde{u}\|_{L^\infty(\Omega)}=\|\widetilde{u}\|_{L^\infty(\Omega_k)},
\]
where
\[
\Omega_k=\bigcup_{\mathbf{i}\in\mathbb{Z}^N_k}D_{\mathbf{t}}(\mathbf{i}+(\overline{Q_{1/2}}\setminus K))
\]
and $\mathbb{Z}^N_k$ is as in \eqref{Ztroncato}. Observe that $\Omega_k$ is a finite union of perforated hyper-rectangles and we have 
\[
\|\widetilde{u}\|_{L^\infty(\Omega)}=\|\widetilde{u}\|_{L^\infty(\Omega_k)}=\max_{\mathbf{i}\in\mathbb{Z}^N_k} \|\widetilde{u}\|_{L^\infty(D_{\mathbf{t}}(\mathbf{i}+(\overline{Q_{1/2}}\setminus K))}.
\] 
Thus, there exists $\mathbf{i}_0\in\mathbb{Z}^N_k$ such that
\[
\|\widetilde{u}\|_{L^\infty(\Omega)}=\|\widetilde{u}\|_{L^\infty(D_{\mathbf{t}}(\mathbf{i}_0+(\overline{Q_{1/2}}\setminus K))}.
\]
By finally defining $u(x)=\widetilde{u}(x+D_\mathbf{t}(\mathbf{i}_0))$, we get that $u$ is the claimed minimizer.
\end{proof}
As in \cite{BraBriPri}, for $p>N$ we can cover the endpoint case $q=\infty$, by a limiting argument.
\begin{teo}[The case $q=\infty$]
\label{teo:maininfty}
Let $N<p<\infty$ and let $\Omega\subseteq\mathbb{R}^N$ be an open set satisfying the assumptions of Theorem \ref{teo:main}.
For every $p<q<\infty$, let $u_q\in W^{1,p}_0(\Omega)$ be the extremal for $\lambda_{p,q}(\Omega)$ provided by Theorem \ref{teo:main}.
Then, the family $\{u_q\}_{q>p}$
is precompact in $W^{1,p}_0(\Omega)$ and every accumulation point is an extremal of $\lambda_{p,\infty}(\Omega)$.
\par
Moreover, if for some diverging sequence $\{q_n\}_{n\in\mathbb{N}}$ and a function $u_\infty\in W^{1,p}_0(\Omega)$, it holds  
\[
\lim_{n\to\infty}\|u_{q_n}-u_\infty\|_{W^{1,p}(\Omega)}=0,
\] 
then we also have 
\[
u_{q_n}^{q_n-1}\ \stackrel{n\to\infty}{\rightharpoonup}\ \delta_{z_\infty},\qquad \mbox{in}\ \mathscr{D}'(\Omega).
\]
Here  $z_\infty$ is the (unique) maximum point of $u_\infty$.
\end{teo}
\begin{proof} 
The proof goes along the lines of \cite[Theorem 5.2]{BraBriPri} with some suitable changes. 
As above, we first observe that $\lambda_p(\Omega)>0$, thanks to Lemma \ref{lm:nonbanale}. Again by \eqref{mara}, we obtain that $\lambda_{p,\infty}(\Omega)>0$, as well. On account of Lemma \ref{lm:inradius} we have $r_{\Omega}<+\infty$: this guarantees that $\{u_q\}_{q>p}$ is a bounded family in $W^{1,p}_0(\Omega)$, as in the proof of \cite[Theorem 5.2]{BraBriPri}. Moreover, we recall that
\begin{equation}\label{preliminare}
\lim_{q\to\infty}\lambda_{p,q}(\Omega)=\lambda_{p,\infty}(\Omega).
\end{equation}
see for instance \cite[Corollary 6.2]{BraPriZag}.
\par
By repeating exactly the same arguments as in \cite[Theorem 5.2]{BraBriPri}, we get that the family $\{u_q\}_{q>p}$ is bounded in $W^{1,p}_0(\Omega)$ and both equicontinuous and bounded in $C^0_{\rm b}(\overline\Omega)$. The latter is the Banach space of continuous and bounded functions over $\overline\Omega$, endowed with the sup norm. 
Thanks to the properties of $u_q$, we know that for every $q>p$ 
\[
\|u_q\|_{L^\infty(D_{\mathbf{t}}(Q_{1/2}\setminus K))}=\|u_q\|_{L^{\infty}(\Omega)}.
\]
Observe that
\[
\int_\Omega |\nabla u_q|^p\,dx=\lambda_{p,q}(\Omega)\,\int_\Omega |u_q|^q\,dx\le \lambda_{p,q}(\Omega)\,\|u_q\|_{L^\infty(\Omega)}^{q-p}\,\int_\Omega |u_q|^p\,dx, \qquad \text{for every}\ q >p.
\]
Thus, by applying the Poincar\'e inequality on the left-hand side, we get in particular
\begin{equation}\label{boundnorma0}
\|u_q\|_{L^\infty(D_{\mathbf{t}}(Q_{1/2}\setminus K))} \geq  \left(\frac{\lambda_{p}(\Omega)}{\lambda_{p,q}(\Omega)}\right)^\frac{1}{q-p}.
\end{equation}
We show that $\{u_q\}_{q>p}$ is precompact in $W^{1,p}_0(\Omega)$,  in the norm topology. Indeed, let $\{q_n\}_{n\in\mathbb{N}}\subseteq (p,+\infty)$ be any sequence diverging to $+\infty$.
By applying the Ascoli-Arzel\`a Theorem and the reflexivity of the space  $W^{1,p}_0(\Omega)$, we have that there exists a subsequence of $\{q_n\}_{n\in\mathbb{N}}$ (not relabelled) and a function $u_{\infty}\in W^{1,p}_0(\Omega)\cap C^0(D_{\mathbf{t}}(\overline{Q_{1/2}}))$  such that  $\{u_{q_n}\}_{n\in\mathbb{N}}$  weakly converges to $u_{\infty}$ in $W^{1,p}_0(\Omega)$ and uniformly on  $D_{\mathbf{t}}(\overline{Q_{1/2}})$. 
\par
Thanks to \eqref{boundnorma0},  the uniform convergence of  $\{u_{q_n}\}_{n\in\mathbb{N}}$ on $D_{\mathbf{t}}(\overline{Q_{1/2}})$ and to \eqref{preliminare}, we get that 
 \begin{equation}\label{boundnorma}
\|u_{\infty}\|_{L^\infty(\Omega)}\geq \|u_\infty\|_{L^\infty(D_{\mathbf{t}}(Q_{1/2}\setminus K))}\geq 1.
\end{equation}
On the other hand, by using the definition of $\lambda_{p,\infty}(\Omega)$ and the weak convergence in $W^{1,p}(\Omega)$ of the sequence $\{u_{q_n}\}_{n\in\mathbb{N}}$, we get
\begin{equation}
\label{boundgrad2}
\lambda_{p,\infty}(\Omega)\,\|u_\infty\|^{p}_{L^{\infty}(\Omega)}\le \int_{\Omega}|\nabla u_\infty|^p\,dx\le \lim_{n\to \infty}\int_{\Omega}|\nabla u_{q_n}|^p\,dx=\lambda_{p,\infty}(\Omega).
\end{equation}
where we have again used \eqref{preliminare}.
By combining \eqref{boundnorma} and  \eqref{boundgrad2},  we deduce
\[
\|u_\infty\|_{L^{\infty}(\Omega)}= 1 \qquad\text{and}\qquad  \lim_{n\to\infty}\|\nabla u_{q_n}\|^p_{L^{p}(\Omega)}= \int_{\Omega}|\nabla u_{\infty}|^p\,dx=\lambda_{p,\infty}(\Omega).
\]
Then $u_\infty$ is an extremal for $\lambda_{p,\infty}(\Omega)$ and, thanks to the weak convergence of $\{\nabla u_{q_n}\}_{n\in\mathbb{N}}$ to $\nabla u_{\infty}$ in $L^p(\Omega)$, we  get  
\[
\lim_{n\to \infty}\|\nabla u_{q_n}-\nabla u_{\infty}\|_{L^{p}(\Omega)}=0.
\]
This ensures the claimed convergence in $W^{1,p}_0(\Omega)$. Notice that repeating these arguments it is also easy to show that any accumulation point of $\{u_q\}_{q>p}$ is in fact an extremal of $\lambda_{p,\infty}(\Omega)$.
\par
Finally, the last part of the theorem follows as in \cite[Theorem 5.2]{BraBriPri}. We just recall that the existence and uniqueness of the point $z_\infty$ follows from \cite[Lemma A.1 \& A.2]{BraBriPri}.
\end{proof}

\begin{rem}
We remark that, for $p>N$, we have 
\[
\mathrm{cap}_p(K;Q_1)>0 \qquad \Longleftrightarrow\qquad K\not=\emptyset.
\]
In particular, we can thus take $K=\{0\}$.
Then, the previous results assure in particular that we have existence of an extremal for the ``pepper set'' $\mathbb{R}^N\setminus\mathbb{Z}^N$, i.e.
\[
\lambda_{p,q}(\mathbb{R}^N\setminus \mathbb{Z}^N),
\]
is attained in $W^{1,p}_0(\mathbb{R}^N\setminus\mathbb{Z}^N)$,
for every $N<p<q\le \infty$.
\end{rem}

\section{Existence of extremals: more general sets}
\label{sec:7}
\begin{figure}
\includegraphics[scale=.25]{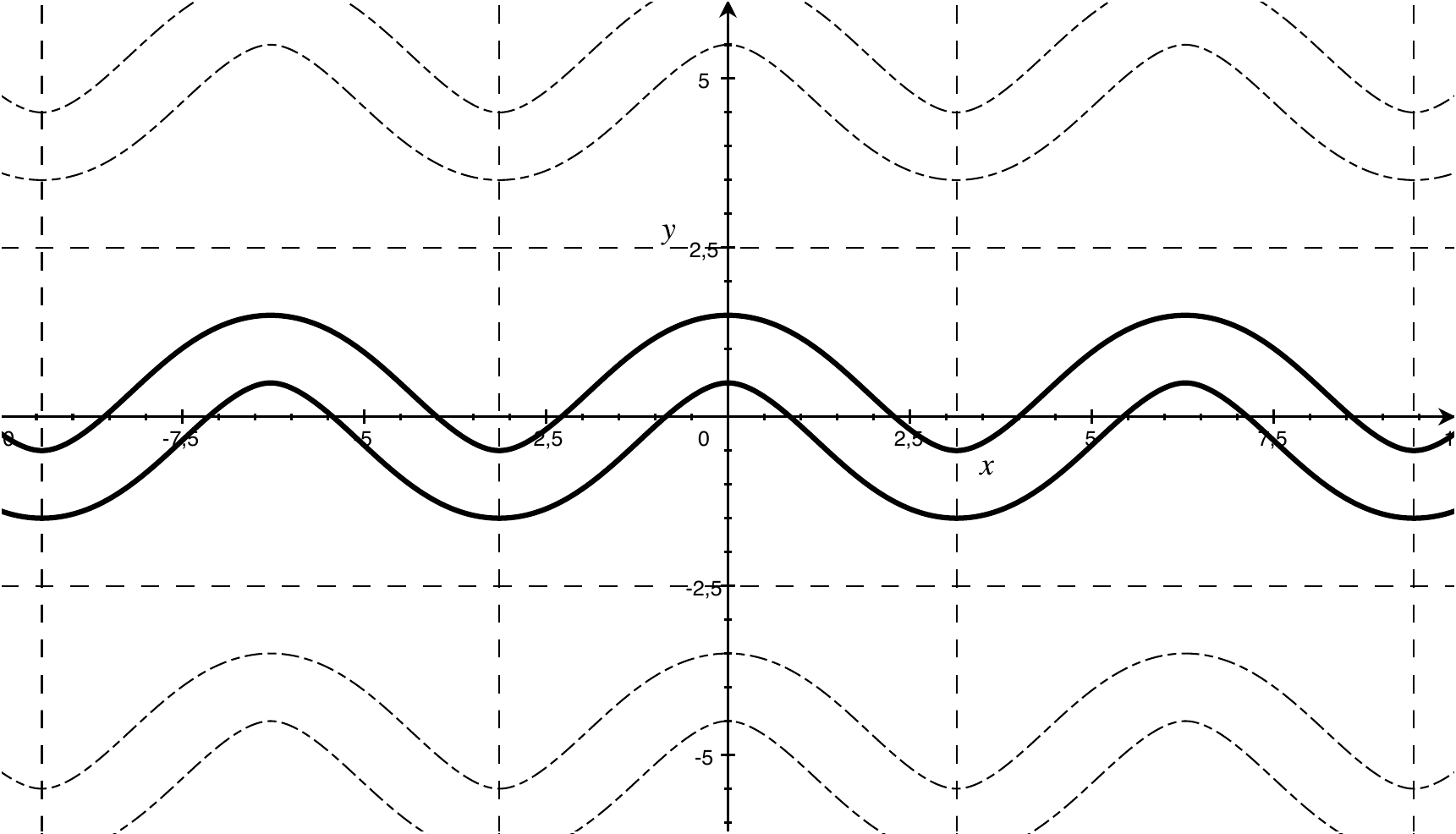}
\caption{In bold line, an example of a two-dimensional open set which is periodic in the first coordinate direction and bounded in the second one. We can extend it periodically in the vertical direction, by adding translated copies of the original set: the new set falls into the realm of Theorem \ref{teo:maininfty}. The rectangles in dashed line represent the periodicity cells.}
\label{fig:bachi}
\end{figure}

We now show how to apply the main result of the paper, in order to get existence of extremals for a more general class of sets: these are bounded in some directions and periodic in the others.
\begin{teo}
\label{teo:pupazzo}
Let $N\ge 2$ and $k\in\{1,\dots,N-1\}$. Let $\Omega\subseteq\mathbb{R}^N$ be an open set with the following properties:
\begin{itemize}
\item[(P)] there exist $t_1,\dots,t_k$ positive numbers such that 
\[
\Omega+t_i\,\mathbf{e}_i=\Omega,\qquad \text{for every}\ i\in\{1,\dots,k\};
\]
\item[(B)] there exists $a>0$ such that 
\[
\Omega\subseteq \mathbb{R}^k\times\left(-\frac{a}{2},\frac{a}{2}\right)^{N-k}.
\]
\end{itemize}
Then, for every $1<p<\infty$ and $q>p$ satisfying \eqref{exponents}, there exists an extremal for $\lambda_{p,q}(\Omega)$.
\end{teo}
\begin{proof}
The strategy of the proof is quite simple, it is better to declare it from the very beginning, in order to assist the reader:
\begin{itemize}
\item we periodically extend the original set $\Omega$ (see Figure \ref{fig:bachi});
\vskip.2cm
\item this new set $\widetilde{\Omega}$ is made of countably many disjoint sets $\Omega_i$ and we can apply Theorems \ref{teo:main} and \ref{teo:maininfty} to $\widetilde\Omega$;
\vskip.2cm
\item we finally appeal to Lemma \ref{lm:disgiunti} and to the fact that each component $\Omega_i$ is just a translated copy of $\Omega$. 
\end{itemize}
We then define the new open set
\[
\widetilde{\Omega}=\bigcup_{i=k+1}^{N} \bigcup_{m\in\mathbb{Z}}\Omega_{i,m},\qquad \mbox{where}\ \Omega_{i,m}=\Omega+2\,m\,a\,\mathbf{e}_i.
\]
This is periodic in every direction $\{\mathbf{e}_1,\dots,\mathbf{e}_N\}$, by construction. More precisely, we have 
\[
\widetilde{\Omega}+t_i\,\mathbf{e}_i=\widetilde\Omega,\qquad \text{for every}\ i\in\{1,\dots,k\},
\]
and
\[
\widetilde{\Omega}+2\,a\,\mathbf{e}_i=\widetilde\Omega,\qquad \text{for every}\ i\in\{k+1,\dots,N\}.
\]
Observe that the translated copies $\Omega_{i,m}$ have been taken in such a way to leave ``enough space'' between each of them. We do this in order to be sure to have enough complement for $\widetilde{\Omega}$: we will be more clear in a moment.
\par
We are going to show that $\widetilde{\Omega}$ fits into the assumptions of Theorems \ref{teo:main} and \ref{teo:maininfty}: in light of what we said above, this will be sufficient to conclude. To this aim, let us define
\[
\overline{\mathbf{t}}=(\overline{t}_1,\dots,\overline{t}_N):=(t_1,\dots,t_k,2a,\dots,2a),
\]
and set
\[
R_{\overline{\mathbf{t}}}=\prod_{i=1}^N\left[-\frac{\overline{t}_i}{2},\frac{\overline{t}_i}{2}\right],\qquad K=D_{\overline{\mathbf{t}}}^{-1}\left(R_{\overline{\mathbf{t}}}\setminus\Omega\right).
\]
By construction, the set $K$ is a compact subset of $\overline{Q_{1/2}}$. Moreover, we have 
\begin{equation}
\label{bloblo}
\mathrm{cap}_p(K;Q_1)>0.
\end{equation}
This can be seen as follows: observe that
\[
\begin{split}
K&\supseteq D_{\overline{\mathbf{t}}}^{-1}\left(\prod_{i=1}^k\left[-\frac{t_i}{2},\frac{t_i}{2}\right]\times \left[\frac{a}{2},a\right]^{N-k}\right)=\left[-\frac{1}{2},\frac{1}{2}\right]^k\times \left[\frac{1}{4},\frac{1}{2}\right]^{N-k}.
\end{split}
\]
The latter obviously has positive $N-$dimensional Lebesgue measure, thus $|K|>0$, as well. Then \eqref{bloblo} follows from the basic inequality
\[
\mathrm{cap}_p(K;Q_1)\ge \lambda_p(Q_1)\,|K|,
\] 
see \cite[Corollary 2.3.4]{Maz}. 
\par
We claim that $\widetilde{\Omega}$ is a $\overline{\mathbf{t}}-$periodically perforated set, generated by $K$ above: we need to show that
\begin{equation}
\label{uguali}
\widetilde{\Omega}=\mathbb{R}^N\setminus\left(\bigcup_{\mathbf{i}\in\mathbb{Z}^N}D_{\overline{\mathbf{t}}}(\mathbf{i}+K)\right).
\end{equation}
By recalling the definition of $K$, we have
\[
D_{\overline{\mathbf{t}}}(\mathbf{i}+K)=D_{\overline{\mathbf{t}}}(\mathbf{i})+(R_{\overline{\mathbf{t}}}\setminus\Omega).
\]
We also observe that 
\[
\mathbb{R}^N=\bigcup_{\mathbf{i}\in\mathbb{Z}^N}D_{\overline{\mathbf{t}}}\left(\mathbf{i}+\overline{Q_{1/2}}\right)=\bigcup_{\mathbf{i}\in\mathbb{Z}^N}\left(D_{\overline{\mathbf{t}}}(\mathbf{i})+R_{\overline{\mathbf{t}}}\right).
\]
Thus, by using Lemma \ref{lm:tennico} below, we can infer that
\[
\begin{split}
\mathbb{R}^N\setminus\left(\bigcup_{\mathbf{i}\in\mathbb{Z}^N}D_{\overline{\mathbf{t}}}(\mathbf{i}+K)\right)&=\bigcup_{\mathbf{i}\in\mathbb{Z}^N}\left(D_{\overline{\mathbf{t}}}(\mathbf{i})+R_{\overline{\mathbf{t}}}\right)\setminus \bigcup_{\mathbf{i}\in\mathbb{Z}^N}\left(D_{\overline{\mathbf{t}}}(\mathbf{i})+R_{\overline{\mathbf{t}}}\setminus\Omega\right)\\
&=\bigcup_{\mathbf{i}\in\mathbb{Z}^N}\left(D_{\overline{\mathbf{t}}}(\mathbf{i})+\left(\Omega\cap R_{\overline{\mathbf{t}}}\right) \right)\\
&=\bigcup_{\mathbf{i}\in\mathbb{Z}^N}\left((D_{\overline{\mathbf{t}}}(\mathbf{i})+\Omega)\cap \left(D_{\overline{\mathbf{t}}}(\mathbf{i})+R_{\overline{\mathbf{t}}}\right) \right).
\end{split}
\]
We write the last union as follows
\[
\begin{split}
\bigcup_{\mathbf{i}\in\mathbb{Z}^N}&\left((D_{\overline{\mathbf{t}}}(\mathbf{i})+\Omega)\cap \left(D_{\overline{\mathbf{t}}}(\mathbf{i})+R_{\overline{\mathbf{t}}}\right) \right)\\
&=\bigcup_{(i_{k+1},\dots,i_N)\in\mathbb{Z}^{N-k}}\bigcup_{\mathbf{i}'\in\mathbb{Z}^k}\left((D_{\overline{\mathbf{t}}}(\mathbf{i}',i_{k+1},\dots,i_N)+\Omega)\cap \left(D_{\overline{\mathbf{t}}}(\mathbf{i}',i_{k+1},\dots,i_N)+R_{\overline{\mathbf{t}}}\right) \right).
\end{split}
\]
Thanks to the periodicity assumption (P) and recalling the definition of $\overline{\mathbf{t}}$, we have 
\[
D_{\overline{\mathbf{t}}}(\mathbf{i}',i_{k+1},\dots,i_N)+\Omega=\Omega+\sum_{j=1}^k t_j\,i_j\,\mathbf{e}_j +\sum_{j=k+1}^{N} 2\,a\,i_j\,\mathbf{e}_j=\Omega+\sum_{j=k+1}^{N} 2\,a\,i_j\,\mathbf{e}_j.
\]
We thus obtain
\[
\begin{split}
\mathbb{R}^N&\setminus\left(\bigcup_{\mathbf{i}\in\mathbb{Z}^N}D_{\overline{\mathbf{t}}}(\mathbf{i}+K)\right)\\
&=\bigcup_{(i_{k+1},\dots,i_N)\in\mathbb{Z}^{N-k}}\bigcup_{\mathbf{i}'\in\mathbb{Z}^k}\left(\left(\Omega+\sum_{j=k+1}^{N} 2\,a\,i_j\,\mathbf{e}_j\right)\cap \left(D_{\overline{\mathbf{t}}}(\mathbf{i}',i_{k+1},\dots,i_N)+R_{\overline{\mathbf{t}}}\right) \right)\\
&=\bigcup_{(i_{k+1},\dots,i_N)\in\mathbb{Z}^{N-k}}\left(\Omega+\sum_{j=k+1}^{N} 2\,a\,i_j\,\mathbf{e}_j\right)\cap\left(\bigcup_{\mathbf{i}'\in\mathbb{Z}^k}\left(D_{\overline{\mathbf{t}}}(\mathbf{i}',i_{k+1},\dots,i_N)+R_{\overline{\mathbf{t}}}\right) \right)
\end{split}
\]
Observe that 
\[
\begin{split}
\bigcup_{\mathbf{i}'\in\mathbb{Z}^k}\biggl(
D_{\overline{\mathbf{t}}}(\mathbf{i}',i_{k+1},\dots,i_N)\bigl.&\bigl.+R_{\overline{\mathbf{t}}}
\biggr)=\mathbb{R}^k\times\prod_{j=k+1}^N\left[-a+2\,a\,i_{j},a+2\,a\,i_j\right]\\
&=\left(\mathbb{R}^k\times[- a, a]^{N-k}\right)+\sum_{j=k+1}^{N} 2\,a\,i_j\,\mathbf{e}_j,
\end{split}
\]
and thus by assumption (B) we get
\[
\begin{split}
\left(\Omega+\sum_{j=k+1}^{N} 2\,a\,i_j\,\mathbf{e}_j\right)&\cap\left(\bigcup_{\mathbf{i}'\in\mathbb{Z}^k}\left(D_{\overline{\mathbf{t}}}(\mathbf{i}',i_{k+1},\dots,i_N)+R_{\overline{\mathbf{t}}}\right) \right)=\Omega+\sum_{j=k+1}^{N} 2\,a\,i_j\,\mathbf{e}_j.
\end{split}
\]
In conclusion, we obtain 
\[
\begin{split}
\mathbb{R}^N\setminus\left(\bigcup_{\mathbf{i}\in\mathbb{Z}^N}D_{\overline{\mathbf{t}}}(\mathbf{i}+K)\right)&=\bigcup_{(i_{k+1},\dots,i_N)\in\mathbb{Z}^{N-k}}\left(\Omega+\sum_{j=k+1}^{N} 2\,a\,i_j\,\mathbf{e}_j\right)\\
&=\bigcup_{j=k+1}^N\bigcup_{m\in\mathbb{Z}} (\Omega+2\,a\,m\,\mathbf{e}_j).
\end{split}
\]
The latter is the definition of $\widetilde\Omega$, thus we established \eqref{uguali}. The proof of the existence of an extremal for $\lambda_{p,q}(\Omega)$ now follows the lines presented at the beginning.
\end{proof}
\begin{lm}
\label{lm:tennico}
With the previous notation, we have that
\[
\begin{split}
\bigcup_{\mathbf{i}\in\mathbb{Z}^N}\left(D_{\overline{\mathbf{t}}}(\mathbf{i})+R_{\overline{\mathbf{t}}}\right)\setminus \bigcup_{\mathbf{i}\in\mathbb{Z}^N}\left(D_{\overline{\mathbf{t}}}(\mathbf{i})+R_{\overline{\mathbf{t}}}\setminus\Omega\right)=\bigcup_{\mathbf{i}\in\mathbb{Z}^N}\left(D_{\overline{\mathbf{t}}}(\mathbf{i})+\left(\Omega\cap R_{\overline{\mathbf{t}}} \right) \right).
\end{split}
\]
\end{lm}
\begin{proof}
The inclusion $\subseteq $ is easier, it is sufficient to use the following general fact: for families of sets $\{A_i'\}_{i\in\mathbb{N}}$ and $\{A_i\}_{i\in\mathbb{N}}$, we have 
\[
\bigcup_{i\in\mathbb{N}} A_i\setminus \bigcup_{i\in\mathbb{N}} A_i'\subseteq\bigcup_{i\in\mathbb{N}} (A_i\setminus A_i').
\]
In order to prove the reverse inclusion $\supseteq $,  we take 
\[
y\in D_{\overline{\mathbf{t}}}(\mathbf{i})+\left(\Omega\cap R_{\overline{\mathbf{t}}} \right),\qquad \text{for some}\ \mathbf{i}\in\mathbb{Z}^N, 
\]
and we wish to prove that
\[
y\not\in  D_{\overline{\mathbf{t}}}(\mathbf{j})+(R_{\overline{\mathbf{t}}}\setminus\Omega),\qquad \text{for every}\, \mathbf{j}\in \mathbb Z^N.
\]
We argue by contradiction and suppose that there exists $\mathbf{j}\in\mathbb{Z}^N$ such that
\[
y\in  D_{\overline{\mathbf{t}}}(\mathbf{j})+(R_{\overline{\mathbf{t}}}\setminus\Omega).
\]
Of course, we have $\mathbf{j}\not=\mathbf{i}$. By observing that the interiors of the two hyper-rectangles $D_{\overline{\mathbf{t}}}(\mathbf{i})+R_{\overline{\mathbf{t}}}$ and $D_{\overline{\mathbf{t}}}(\mathbf{j})+R_{\overline{\mathbf{t}}}$ are disjoint and only their boundaries can intersect, it must result that:
\begin{itemize}
\item $|\mathbf{i}-\mathbf{j}|_{\ell^\infty}=1$, i.e. the two hyper-rectangles are adjacent;
\vskip.2cm
\item $y\in (D_{\overline{\mathbf{t}}}(\mathbf{i})+\partial R_{\overline{\mathbf{t}}})\cap (D_{\overline{\mathbf{t}}}(\mathbf{j})+\partial R_{\overline{\mathbf{t}}})$.
\end{itemize}
Then 
\[
D_{\overline{\mathbf{t}}}(\mathbf{i})+w=y=D_{\overline{\mathbf{t}}}(\mathbf{j})+z,
\]
with $w\in \partial R_{\overline{\mathbf{t}}}\cap \Omega$ and $z\in \partial R_{\overline{\mathbf{t}}}\setminus \Omega$. 
In particular, we have 
\begin{equation}
\label{visgini}
z=w+\sum_{m=1}^N \overline{t}_i\,(j_m-i_m)\,\mathbf{e}_m.
\end{equation}
We highlight another property of $y$: recall that $\Omega\subseteq\mathbb{R}^k\times(-a/2,a/2)^{N-k}$. This implies that 
\[
-\frac{a}{2}<\langle w,\mathbf{e}_n\rangle<\frac{a}{2},\qquad\text{for every}\ n\in\{k+1,\dots,N\}.
\]
By observing that (recall that $\overline{t}_n=2\,a$, for $n\in\{k+1,\dots,N\}$)
\[
\langle z,\mathbf{e}_n\rangle=\langle w,\mathbf{e}_n\rangle+2\,a\,(i_n-j_n),
\]
we then obtain
\[
-\frac{a}{2}+2\,a\,(i_n-j_n)<\langle z,\mathbf{e}_n\rangle<\frac{a}{2}+2\,a\,(i_n-j_n),\qquad\text{for every}\ n\in\{k+1,\dots,N\}.
\]
If for some $n\in\{k+1,\dots,N\}$ we had $j_n=i_n+1$, the previous upper bound would give 
\[
\langle z,\mathbf{e}_n\rangle<-\frac{3}{2}\,a,
\]
contradicting the fact that $z\in R_{\overline{\mathbf{t}}}$. Analogously, if we had $j_n=i_n-1$, the lower bound now would give
\[
\langle z,\mathbf{e}_n\rangle>\frac{3}{2}\,a,
\]
which is again not possible. This in turn implies that we must have $i_n=j_n$ for every $n\in\{k+1,\dots,N\}$.
\par
Then, by \eqref{visgini} we get 
\[
z=w+\sum_{m=1}^k \overline{t}_i\,(j_m-i_m)\,\mathbf{e}_m.
\]
By recalling that $w$ in particular belongs to $\Omega$ and using the periodicity of $\Omega$ in the first $k$ coordinate direction, we get that the right-hand side above gives a point belonging to $\Omega$, as well. In other words, we get $z\in\Omega$, which is a contradiction.
\end{proof}
We present a simple example, in order to give a flavour of the range of applicability of the previous result.
\begin{exa}
For a given $R>0$, let us consider $\Sigma\subseteq\mathbb{R}^3$ the image the following three-dimensional curve
\[
\gamma(t)=(t,R\,\cos t,R\,\sin t),\qquad \text{for}\ t\in\mathbb{R}.
\]
We then define 
\[
\Sigma_r=\left\{x\in\mathbb{R}^3\, :\, \mathrm{dist}(x,\Sigma)<r\right\},
\]
for some $r<R$ (see Figure \ref{fig:fusillone}). 
\begin{figure}
\includegraphics[scale=.25]{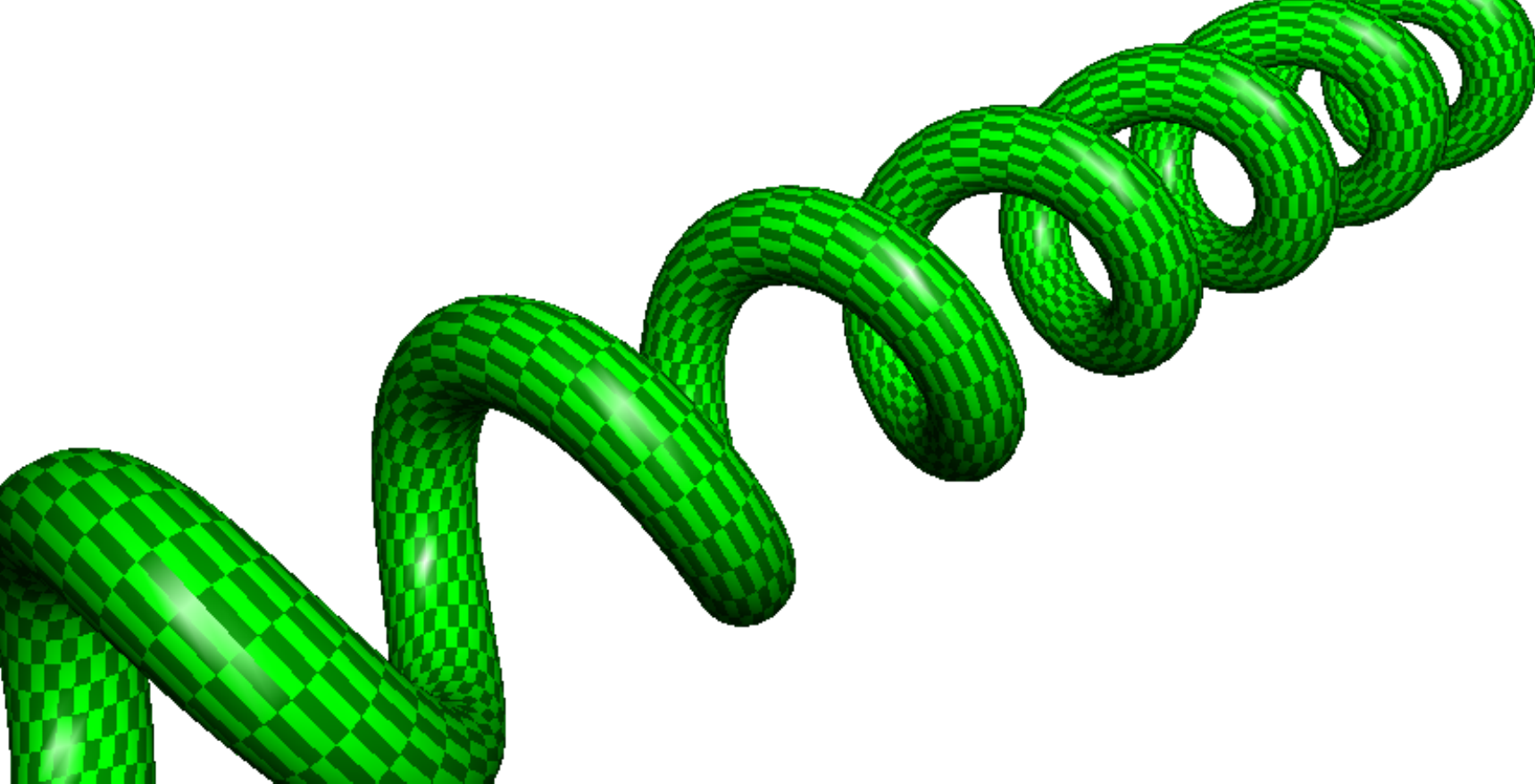}
\caption{The tubular neighborhood of a helical spring: by Theorem \ref{teo:pupazzo}, for this set we have existence of extremals for $\lambda_{p,q}$, with $q>p$.}
\label{fig:fusillone}
\end{figure}
Observe that this set satisfies the assumption of Theorem \ref{teo:pupazzo} with $N=3$, $k=1$ and $t_1=2\,\pi$, i.e. we have 
\[
\Sigma_r+2\,\pi\,\mathbf{e}_1=\Sigma_r.
\]
Accordingly, we get existence of an extremal for $\lambda_{p,q}(\Sigma_r)$, for every $1<p<\infty$ and every $q$ satisfying \eqref{exponents}.
\end{exa}
As a particular case of the previous theorem, we can recover an existence result due to Esteban (see \cite[Theorems 1 and 6]{Es}), in the case of a power--type nonlinearity. This deals with cylindrical--type sets, i.e. open sets of the form $\mathbb{R}^k\times \omega$, with $\omega\subseteq \mathbb{R}^{N-k}$ open {\it bounded} set.
\begin{coro}
\label{coro:Es}
Let $N\ge 2$ and $k\in\{1,\dots,N-1\}$. Let $\omega\subseteq \mathbb{R}^{N-k}$ be an open bounded set. Then, for every $1<p<\infty$ and every $q>p$ satisfying \eqref{exponents}, there exists an extremal for $\lambda_{p,q}(\mathbb{R}^k\times\omega)$.
\end{coro}

\section{Breaking the periodicity: examples}
\label{sec:8}

In this section, we fix $0<r<1/2$ and consider the model open set
\[
\Omega=\mathbb{R}^N\setminus \left(\bigcup_{\mathbf{i}\in\mathbb{Z}^N}\overline{B_r}(\mathbf{i})\right).
\]
This is a $\mathbf{t}-$periodically perforated set, generated by $K$, with
\[
\mathbf{t}=(1,\dots,1)\qquad\text{and}\qquad K=\overline{B_r}.
\]
By Theorems \ref{teo:main} and \ref{teo:maininfty}, we have existence of extremals for $\lambda_{p,q}(\Omega)$, for every $1<p<q$ satisfying \eqref{exponents}. We will show in the next two subsections that a slight modification of $\Omega$, which breaks the periodicity, may or may not lead to a loss of existence of extremals.

\subsection{Existence may be lost}
\label{exa:perforato_exa}
We fix a second radius $r<R<1/2$ and define
\[
\Omega_R= \mathbb{R}^N\setminus \left[\left(\bigcup_{\mathbf{i}\in\mathbb{Z}^N\setminus\{\mathbf{0}\}}\overline{B_r(\mathbf{i})}\right) \cup \overline{B_R}\right],
\]  
see Figure \ref{fig:perforato}.
\begin{figure}
\includegraphics[scale=.25]{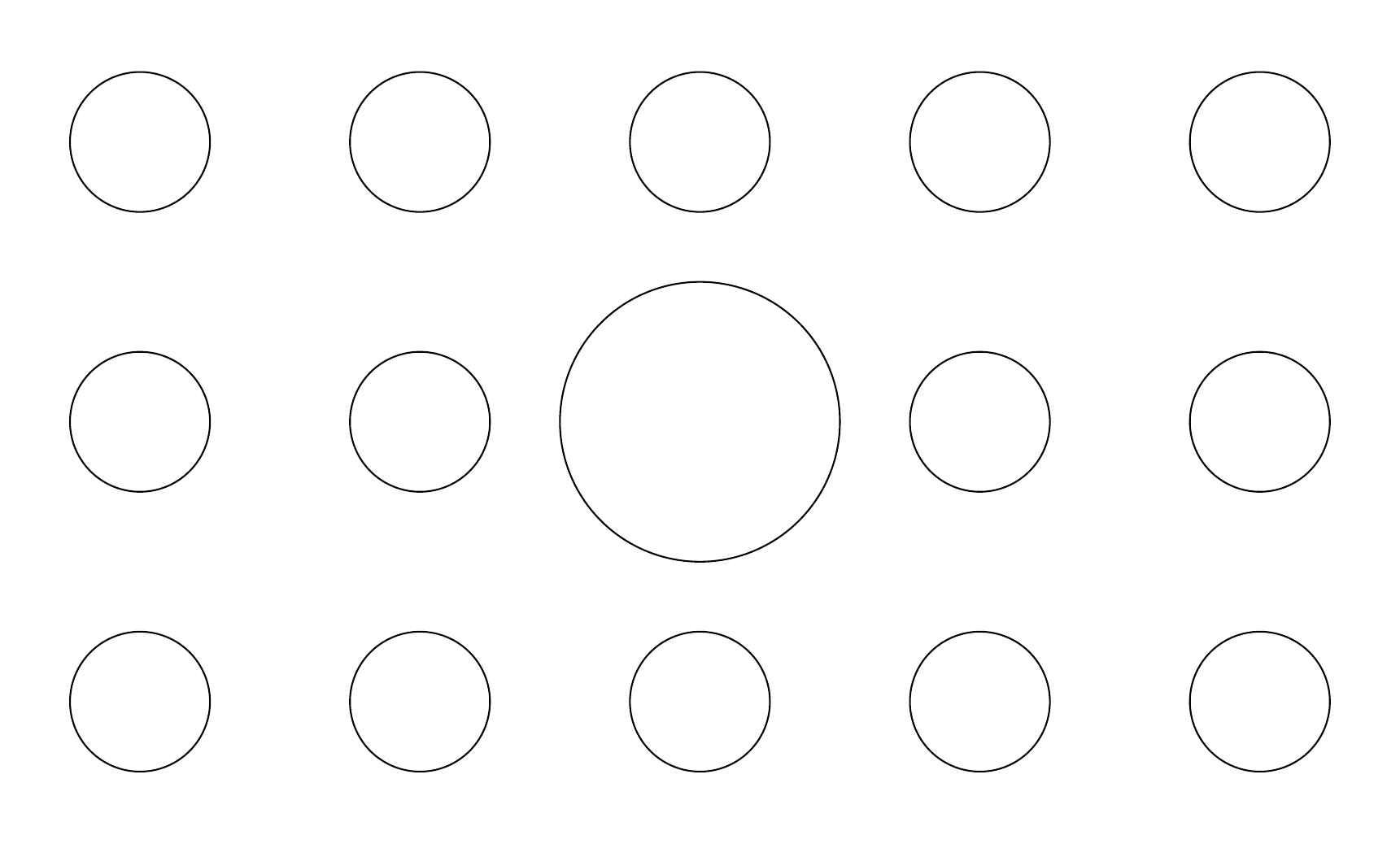}
\caption{The set $\Omega_R$ of Example \ref{exa:perforato_exa}: we remove from $\mathbb{R}^N$ a periodic array of equal balls, except for a larger one, centered at the origin for example. For this set $\lambda_{p,q}$ fails to have extremals.}
\label{fig:perforato}
\end{figure}
We note that  the open set $\Omega_R$ does not satisfy the assumptions of Theorems \ref{teo:main} and \ref{teo:maininfty}. We claim that $\lambda_{p,q}(\Omega_R)$ {\it is not} attained in $W^{1,p}_0(\Omega_R)$, when $p<q$ satisfy \eqref{exponents}.
To this aim, we start by noticing that 
\begin{equation}
\label{uguafori}
\lambda_{p,q}(\Omega_R)=\lambda_{p,q}(\Omega).
\end{equation}
Indeed, the fact that $\lambda_{p,q}(\Omega)\le \lambda_{p,q}(\Omega_R)$ simply follows by the inclusion $\Omega_R\subseteq \Omega$ and the monotonicity of $\lambda_{p,q}$. On the other hand, for every $\varepsilon>0$ there exists $\varphi_\varepsilon\in C^\infty_0(\Omega)$  such that
\[
\lambda_{p,q}(\Omega)+\varepsilon\ge \int_\Omega |\nabla \varphi_\varepsilon|^p\,dx\qquad\text{and}\qquad \| \varphi_\varepsilon\|_{L^q(\Omega)}=1.
\]
Since $\varphi_\varepsilon$ is compactly supported, there exists $n_\varepsilon\in\mathbb{N}$ such that its support is contained in the open set $\Omega\cap \{ x_1>-n_\varepsilon\}$. Therefore, thanks to the periodicity of $\Omega$ and to the translation invariance of the sharp Poincar\'e-Sobolev constant, we have
\[
\begin{split}
\lambda_{p,q}(\Omega)+\varepsilon&\ge \int_\Omega |\nabla \varphi_\varepsilon|^p\,dx\\
&\ge \lambda_{p,q}(\Omega\cap \{ x_1> -n_\varepsilon\})=\lambda_{p,q}(\Omega\cap \{ x_1> 1\}).
\end{split}
\]
On the other hand, by observing that $\Omega\cap \{ x_1> 1\}=\Omega_R\cap \{ x_1> 1\}$ and using the monotonicity of $\lambda_{p,q}$ with respect to the set inclusion, we have
\[
\lambda_{p,q}(\Omega\cap \{ x_1> 1\})=\lambda_{p,q}(\Omega_R\cap \{ x_1> 1\})\geq \lambda_{p,q}(\Omega_R).
\]
We thus get
\[
\lambda_{p,q}(\Omega)+\varepsilon\ge \lambda_{p,q}(\Omega_R)
\]
and by the arbitrariness of $\varepsilon>0$, we deduce that $\lambda_{p,q}(\Omega)\ge \lambda_{p,q}(\Omega_R)$, as well.
\par 
Once we have \eqref{uguafori} at our disposal, it is not difficult to show that  $\lambda_{p,q}(\Omega_R)$ does not admit an extremal. Indeed, any positive extremal $u$ for $\lambda_{p,q}(\Omega_R)$ would be an extremal for $\lambda_{p,q}(\Omega)$, as well, thanks to \eqref{uguafori}. In particular, $u\not\equiv0$ would be a $p-$superharmonic function in $\Omega$, identically vanishing on the set $B_{R}\setminus B_{r}$, which has positive measure. This would violate the minimum principle.

\subsection{Existence may persist}
\label{exa:perforato_exa2}

With the previous notation, we fix now a smaller radius $0<\rho<r$ and define
\[
\Omega_\rho= \mathbb{R}^N\setminus \left[\left(\bigcup_{\mathbf{i}\in\mathbb{Z}^N\setminus\{\mathbf{0}\}}\overline{B_r(\mathbf{i})}\right) \cup \overline{B_\rho}\right].
\]  
In other words, we break the periodicity of $\Omega$
by digging this time a smaller hole at the origin. We claim that $\lambda_{p,q}(\Omega_\rho)$ {\it is now attained} in $W^{1,p}_0(\Omega_\rho)$, when $p<q$ satisfy \eqref{exponents}. We will focus for simplicity on the case $p=2$, for which we can appeal to an elementary argument. If we set 
\[
\mathcal{E}_{2,q}(\Omega_\rho):=\sup_{R>0}\lambda_{2,q}\big(\Omega_\rho\setminus \overline{B_R}\big),
\]
we have that (see Proposition \ref{prop:esistenzadelcazzo}), it is sufficient to prove that 
\begin{equation}
\label{sottosotto}
\lambda_{2,q}(\Omega_\rho)<\mathcal{E}_{2,q}(\Omega_\rho),
\end{equation}
in order to get existence of extremals. 
We first observe that 
\[
\mathcal{E}_{2,q}(\Omega_\rho)=\lambda_{2,q}(\Omega),
\]
by Lemma \ref{lm:skorpion} below. This in turn implies the desired property \eqref{sottosotto}.
Indeed, by monotonicity with respect to the set inclusion, we have 
\[
\lambda_{2,q}(\Omega_\rho)\le \lambda_{2,q}(\Omega).
\]
Moreover, since $\lambda_{2,q}(\Omega)$ is attained by some positive function $u\in W^{1,2}_0(\Omega)$, thanks to Theorem \ref{teo:main}, we can not have equality. Otherwise, the function $u$ would be an extremal for $\Omega_\rho$, as well, thus violating again the minimum principle (i.e. it would vanish on the set $B_r\setminus B_\rho$).
\vskip.2cm\noindent
\begin{lm}
\label{lm:skorpion}
With the notation above, we have
\begin{equation}
\label{energiainfinito}
\mathcal{E}_{2,q}(\Omega_\rho)=\mathcal{E}_{2,q}(\Omega)=\lambda_{2,q}(\Omega).
\end{equation} 
\end{lm}
\begin{proof}
We notice that  
\[
\lambda_{2,q}\big(\Omega_\rho\setminus \overline{B_R})=\lambda_{2,q}\big(\Omega\setminus \overline{B_R}),\quad \hbox{ for every } R>r,
\]
thus we get $\mathcal{E}_{2,q}(\Omega_\rho)=\mathcal{E}_{2,q}(\Omega)$. 
Accordingly, proving \eqref{energiainfinito} boils down to showing that 
\[
\lambda_{2,q}(\Omega)=\mathcal{E}_{2,q}(\Omega).
\]
We clearly have that  $\lambda_{2,q}(\Omega)\le \mathcal{E}_{2,q}(\Omega)$, thanks to the  domain monotonicity (just observe that $\Omega\setminus\overline{B_R}\subseteq \Omega$, for every $R>0$). For the converse inequality,  let $ \varepsilon>0$ and let     $m_{\varepsilon}\in \mathbb N$ be  such that  
\begin{equation}
\label{maggioro1}
\mathcal{E}_{2,q}(\Omega) \leq \lambda_{2,q}\big(\Omega\setminus \overline{B_{m_{\varepsilon}}})+\varepsilon.
\end{equation}
We also take $\varphi_\varepsilon\in C^\infty_0(\Omega)$ such that
\[
\lambda_{2,q}(\Omega)+\varepsilon\ge \int_{\Omega} |\nabla \varphi_\varepsilon|^2\,dx\qquad\text{and}\qquad \| \varphi_\varepsilon\|_{L^q(\Omega)}=1.
\]
As in the previous subsection, since $\varphi_\varepsilon$ is compactly supported, we have that its support is contained in the open set $\Omega\cap \{ x_1>-n_\varepsilon\}$, for some $n_\varepsilon\in\mathbb{N}$ large enough. Therefore, we get
\[
\begin{split}
\lambda_{2,q}(\Omega)+\varepsilon&\ge \int_{\Omega} |\nabla \varphi_\varepsilon|^2\,dx\ge \lambda_{2,q}(\Omega\cap \{ x_1> -n_\varepsilon\})=\lambda_{2,q}(\Omega\cap \{ x_1> m_{\varepsilon}\}).\\
\end{split}
\]
The last equality follows from the periodicity of $\Omega$ and the translation invariance of the sharp Poincar\'e-Sobolev constant. If we now observe that $\Omega\cap \{ x_1> m_{\varepsilon}\}\subseteq \Omega\setminus \overline{B_{m_{\varepsilon}}}$, from the previous estimate we get
\[
\lambda_{2,q}(\Omega)+\varepsilon\geq  \lambda_{2,q}\big(\Omega\setminus \overline{B_{m_{\varepsilon}}})\geq \mathcal{E}_{2,q}(\Omega) -\varepsilon,
\]
where we also used \eqref{maggioro1}.
By the arbitrariness of $\varepsilon>0$, we deduce that $\lambda_{2,q}(\Omega)\ge \mathcal{E}_{2,q}(\Omega)$, as well.
\end{proof}

\appendix

\section{Embedding constants at infinity}
\label{sec:A}

In what follows, for an open set $\Omega\subseteq\mathbb{R}^N$ we use the following notation
\[
\mathcal{S}_{p,q}(\Omega)=\Big\{u\in W^{1,p}_0(\Omega)\, :\, \|u\|_{L^q(\Omega)}=1\Big\}.
\]
We also define its {\it Poincar\'e-Sobolev constant at infinity}
\[
\mathcal{E}_{p,q}(\Omega):=\sup_{R>0}\lambda_{p,q}\big(\Omega\setminus \overline{B_R}\big).
\]
The next result asserts that, for sequences of normalized functions whose $p-$Dirichlet energy is below the threshold ``at infinity'' $\mathcal{E}_{p,q}(\Omega)$, compactness in $L^q$ can not be completely lost. More precisely:
\begin{lm}
\label{lm:avanzo}
Let $1<p<\infty$ and let $q\ge p$ be a finite exponent verifying \eqref{exponents}. Let $\Omega\subseteq\mathbb{R}^N$ be an open unbounded set such that $\mathcal{E}_{p,q}(\Omega)>0$. Let $\{u_n\}_{n\in\mathbb{N}}\subseteq \mathcal{S}_{p,q}(\Omega)$ be a sequence with the following properties:
\begin{itemize}
\item $u_n$ weakly converges in $W^{1,p}(\Omega)$ to some function $u\in W^{1,p}_0(\Omega)$;
\vskip.2cm
\item there exist $M<\mathcal{E}_{p,q}(\Omega)$ and $n_0\in\mathbb{N}$ such that 
\[
\int_\Omega |\nabla u_n|^p\,dx\le M,\qquad \text{for every}\ n\ge n_0.
\]
\end{itemize}
Then, there exists $0<\mathfrak{d}=\mathfrak{d}(M/\mathcal{E}_{p,q}(\Omega))<1$ such that $\|u\|_{L^q(\Omega)}\ge \mathfrak{d}$. Moreover, the constant $\mathfrak{d}$ is such that we have 
\[
\lim_{\mathcal{E}_{p,q}(\Omega)\to+\infty} \mathfrak{d}=1.
\]
\end{lm}
\begin{proof}
Let us suppose at first that $\mathcal{E}_{p,q}(\Omega)<+\infty$.
We set 
\[
\theta_0=\frac{M}{\mathcal{E}_{p,q}(\Omega)}<1,
\]
and choose $R_0>0$ such that 
\[
\frac{M}{\lambda_{p,q}(\Omega\setminus\overline{B_{R}})}\leq \sqrt{\theta_0}=:\theta_1,
\qquad \text{for every}\ R\ge R_0.
\]
Observe that $\theta_0<\theta_1<1$.
Such a radius $R_0$ exists thanks to the assumption and the definition of $\mathcal{E}_{p,q}(\Omega)$. We also take $R_1>0$ such that 
\[
\int_{\Omega\setminus B_{R_1}} |\nabla u|^p\,dx\le M\,\theta_2,\qquad \text{where}\ \theta_2=\theta_0^p\,\left(\frac{1}{\theta_1^\frac{1}{p}}-1\right)^p,
\]
a choice which is feasible, since $\nabla u\in L^p(\Omega;\mathbb{R}^N)$.
We set $R_*=\max\{R_0,R_1\}$. For every $r>0$, we take a cut-off function Lipschitz function $\eta$ such that 
\[
0\le \eta\le 1,\qquad \eta\equiv 0\ \text{on}\ \overline{B_{R_*}},\qquad \eta\equiv 1 \mbox{ on }\ \mathbb{R}^N\setminus B_{R_*+1},\qquad \|\nabla \eta\|_{L^\infty}=1.
\]
We then have, thanks to the properties of $\eta$ 
\[
\begin{split}
\left(\int_{\Omega\setminus B_{R_*+1}} |u_n-u|^q\,dx\right)^\frac{1}{q}&\le \left(\int_{\Omega\setminus B_{R_*}} |(u_n-u)\,\eta|^q\,dx\right)^\frac{1}{q}\\
&\le \left(\frac{1}{\lambda_{p,q}(\Omega\setminus\overline{B_{R_*}})}\right)^\frac{1}{p}\,\left(\int_{\Omega\setminus B_{R_*}} |\nabla u_n-\nabla u|^p\,dx\right)^\frac{1}{p}\\
&+\left(\frac{1}{\lambda_{p,q}(\Omega\setminus\overline{B_{R_*}})}\right)^\frac{1}{p}\,\left(\int_{B_{R_*+1}\setminus B_{R_*}} |u_n-u|^p\,dx\right)^\frac{1}{p}.
\end{split}
\]
In particular, by using that
\[
\begin{split}
\left(\int_{\Omega\setminus B_{R_*}} |\nabla u_n-\nabla u|^p\,dx\right)^\frac{1}{p}&\le \left(\int_{\Omega\setminus B_{R_*}} |\nabla u_n|^p\,dx\right)^\frac{1}{p}+\left(\int_{\Omega\setminus B_{R_*}} |\nabla u|^p\,dx\right)^\frac{1}{p}\\
&\le M^\frac{1}{p}+M^\frac{1}{p}\,\theta_2^\frac{1}{p},
\end{split}
\]
we get
\begin{equation}
\label{differenzapq}
\begin{split}
\left(\int_{\Omega\setminus B_{R_*+1}} |u_n-u|^q\,dx\right)^\frac{1}{q}&\le \theta_1^\frac{1}{p}\,\left(1+\theta_2^\frac{1}{p}\right)\\
&+\left(\frac{|B_{R_*+1}\setminus B_{R_*}|^\frac{q-p}{q}}{\lambda_p(\Omega\setminus\overline{B_{R_*}})}\right)^\frac{1}{p}\,\left(\int_{B_{R_*+1}\setminus B_{R_*}} |u_n-u|^q\,dx\right)^\frac{1}{q}.
\end{split}
\end{equation}
Observe that we also used H\"older's inequality on the last term.
By using the compact embedding $W^{1,p}(B_{R^*+1})\hookrightarrow L^q(B_{R^*+1})$, we have
\[
\lim_{n\to\infty} \|u_n-u\|_{L^q(B_{R_*+1})}=0.
\]
Thus, for every $\varepsilon>0$, there exists $n_\varepsilon$ such that 
\[
\|u_n-u\|_{L^q(B_{R_*+1})}< \varepsilon,\qquad \mbox{ for every } n\ge n_\varepsilon.
\]
From \eqref{differenzapq}, for every $n\ge n_\varepsilon$ we have 
\[
\left(\int_{\Omega\setminus B_{R_*+1}} |u_n-u|^q\,dx\right)^\frac{1}{q}\le\theta_1^\frac{1}{p}\,\left(1+\theta_2^\frac{1}{p}\right)+
\varepsilon\,\left(\frac{|B_{R_*+1}\setminus B_{R_*}|^\frac{q-p}{q}}{\lambda_p(\Omega\setminus\overline{B_{R_*}})}\right)^\frac{1}{p}\,.
\]
This in turn gives for every $n\ge n_\varepsilon$
\[
\|u_n-u\|_{L^q(\Omega)}\le\varepsilon+ \theta_1^\frac{1}{p}\,\left(1+\theta_2^\frac{1}{p}\right) +\varepsilon\,\left(\frac{|B_{R_*+1}\setminus B_{R_*}|^\frac{q-p}{q}}{\lambda_p(\Omega\setminus\overline{B_{R_*}})}\right)^\frac{1}{p}.
\]
By Minkowski's inequality, we get
\[
\begin{split}
\|u\|_{L^q(\Omega)}&\ge \|u_n\|_{L^q(\Omega)}-\|u_n-u\|_{L^q(\Omega)}\\
&\ge \left[1-\theta_1^\frac{1}{p}\,\left(1+\theta_2^\frac{1}{p}\right)\right]-\varepsilon\,\left[1+\left(\frac{|B_{R_*+1}\setminus B_{R_*}|^\frac{q-p}{q}}{\lambda_p(\Omega\setminus\overline{B_{R_*}})}\right)^\frac{1}{p}\right].
\end{split}
\]
This lower bound holds for every $\varepsilon>0$. By taking the limit as $\varepsilon$ goes to $0$, we get
\[
\|u\|_{L^q(\Omega)}\ge 1-\theta_1^\frac{1}{p}\,\left(1+\theta_2^\frac{1}{p}\right)=:\mathfrak{d}.
\]
By recalling the definition of $\theta_1$ and $\theta_2$, this is the same as
\[
\mathfrak{d}=1-\theta_0^\frac{1}{2\,p}\,(1-\theta_0)-\theta_0=(1-\theta_0)\,\left(1-\theta_0^\frac{1}{2\,p}\right),
\]
and it is not difficult to see that this quantity is positive.
The last statement, easily follows by observing that
\[
\lim_{\mathcal{E}_{p,q}(\Omega)\to +\infty} \theta_0=0.
\] 
At last, the case when $\mathcal{E}_{p,q}(\Omega)=+\infty$ can be treated with the same argument, by  choosing first any $0<\theta_0<1$ and then by letting $\theta_0$ go to $0$. Accordingly, we can show that in this case we have $\|u\|_{L^q(\Omega)}\ge 1$, as expected. We leave the details to the reader. 
\end{proof}
Thanks to Lemma \ref{lm:avanzo}, we get an easy condition to have existence of extremals for $\lambda_{p,q}$, in some cases: this is the content of the next result. We point out that it {\it does not} apply to periodically perforated sets (see Remark \ref{rem:noperiodico} below).
For this reason, we will focus on the case $p=2$ only: the general case will be treated in a forthcoming paper. 
\begin{prop}
\label{prop:esistenzadelcazzo}
Let $p=2$ and $q>2$ be an exponent satisfying \eqref{exponents}. Let $\Omega\subseteq\mathbb{R}^N$ be an open set such that 
\begin{equation}
\label{minoratoapp}
\lambda_{2,q}(\Omega)<\mathcal{E}_{2,q}(\Omega).
\end{equation}
Then there exists an extremal for $\lambda_{2,q}(\Omega)$.
\end{prop}
\begin{proof}
Let $\{u_n\}_{n\in\mathbb{N}}\subseteq W^{1,2}_0(\Omega)$ be a minimizing sequence for $\lambda_{2,q}(\Omega)$, such that
\[
u_n\in \mathcal{S}_{2,q}(\Omega)\quad \text{and}\quad  \int_\Omega |\nabla u_n|^2\,dx\le \lambda_{2,q}(\Omega)+\frac{1}{n+1},\qquad \text{for every}\ n\in\mathbb{N}.
\]
Up to pass to a subsequence, we can suppose that there exists $u\in W^{1,2}_0(\Omega)$ such that $\{u_n\}_{n\in\mathbb{N}}$ converges weakly in $W^{1,2}(\Omega)$ to $u$. We can also suppose that the sequence converges almost everywhere in $\Omega$ to $u$ (see, for example, \cite[Theorem 8.6]{LiLo}). Thanks to the assumption \eqref{minoratoapp}, we have that $\{u_n\}_{n\in\mathbb{N}}$ satisfies the hypotheses of Lemma \ref{lm:avanzo}. This in particular gives that $u\not\equiv 0$. We are going to show that $u$ is the desired extremal. Indeed, observe that we have
\[
\int_\Omega |\nabla u_n|^2\,dx=\int_\Omega |\nabla u|^2\,dx+\int_\Omega |\nabla u_n-\nabla u|^2\,dx-2\,\int_\Omega \langle \nabla u_n-\nabla u,\nabla u\rangle\,dx,
\]
and the last term converges to $0$, as $n$ goes to $\infty$, thanks to the weak convergence of the gradients. Moreover, again by the Brezis-Lieb Lemma we have 
\[
\int_\Omega |u_n|^q\,dx=\int_\Omega |u|^q\,dx+\int_\Omega |u_n-u|^q\,dx+o(1),\qquad \text{as $n$ goes to}\ \infty.
\]
By using these informations, for $n$ going to $\infty$ we get
\[
\begin{split}
o(1)&=\int_\Omega |\nabla u_n|^2\,dx-\lambda_{2,q}(\Omega)\,\left(\int_\Omega |u_n|^q\,dx\right)^\frac{2}{q}\\
&=\int_\Omega |\nabla u|^2\,dx+\int_\Omega |\nabla u_n-\nabla u|^2\,dx-2\,\int_\Omega \langle \nabla u_n-\nabla u,\nabla u\rangle\,dx\\
&-\lambda_{2,q}(\Omega)\,\left(\int_\Omega |u|^q\,dx+\int_\Omega |u_n-u|^q\,dx+o(1)\right)^\frac{2}{q}\\
&\ge \int_\Omega |\nabla u|^2\,dx+\int_\Omega |\nabla u_n-\nabla u|^2\,dx+o(1)\\
&-\lambda_{2,q}(\Omega)\,\left(\int_\Omega |u|^q\,dx+o(1)\right)^\frac{2}{q}-\lambda_{2,q}(\Omega)\,\left(\int_\Omega |u_n-u|^q\,dx\right)^\frac{2}{q},
\end{split}
\]
thanks to the subadditivity of the power $t\mapsto t^{2/q}$. By observing that
\[
\int_\Omega |\nabla u_n-\nabla u|^2\,dx\ge \lambda_{2,q}(\Omega)\,\left(\int_\Omega |u_n-u|^q\,dx\right)^\frac{2}{q},
\]
we thus obtain
\[
o(1)\ge \int_\Omega |\nabla u|^2\,dx-\lambda_{2,q}(\Omega)\,\left(\int_\Omega |u|^q\,dx+o(1)\right)^\frac{2}{q}.
\]
By taking the limit as $n$ goes to $\infty$, this in turn implies
\[
0\ge \int_\Omega |\nabla u|^2\,dx-\lambda_{2,q}(\Omega)\,\left(\int_\Omega |u|^q\,dx\right)^\frac{2}{q}.
\]
Since the converse inequality holds true by the very definition of $\lambda_{2,q}(\Omega)$, we get that
\[
\int_\Omega |\nabla u|^2\,dx=\lambda_{2,q}(\Omega)\,\left(\int_\Omega |u|^q\,dx\right)^\frac{2}{q}.
\]
By recalling that $u\not\equiv 0$, we finally end up with the existence of an extremal.
\end{proof}
\begin{rem}
\label{rem:noperiodico}
It is useful to remark that Theorem \ref{teo:main} can not be obtained as a consequence of Proposition \ref{prop:esistenzadelcazzo}, in general. More precisely, we can not hope to prove Theorem \ref{teo:main} in full generality, by exploiting the previous argument. As a simple example, we can take 
\[
\Omega=\mathbb{R}^N\setminus \left(\bigcup_{\mathbf{i}\in\mathbb{Z}^N}\overline{B_r(\mathbf{i})}\right),\qquad 0<r<\frac{1}{2}.
\]
It is not difficult to see that for this set the condition \eqref{minoratoapp} crucially fails, since we have seen in Lemma \ref{lm:skorpion} that
\[
\lambda_{2,q}(\Omega)=\mathcal{E}_{2,q}(\Omega).
\]
Nevertheless, we have existence of extremals by Theorem \ref{teo:main}.
\end{rem}


\medskip



\begin{thebibliography}{100}



\bibitem{AmTo} C. J. Amick, J. F. Toland, Nonlinear elliptic eigenvalue problems on an infinite strip -- global theory of bifurcation and asymptotic bifurcation, Math. Ann., {\bf 262} (1983), 313--342.

\bibitem{BBT} J. L. Bona, D. K. Bose, R. E. L. Turner, Finite-amplitude steady waves in stratified fluids,
J. Math. Pures Appl. (9), {\bf 62} (1983), 389--439.

\bibitem{BozBra2} F. Bozzola, L. Brasco, Variations on the capacitary inradius, Discrete Contin. Dyn. Syst. Ser. S, (2025)

\bibitem{BozBra} F. Bozzola, L. Brasco, Capacitary inradius and Poincar\'e-Sobolev inequalities, ESAIM Control Optim. Calc. Var., {\bf 31} (2025), Paper No. 26, 38 pp.

\bibitem{BozBra0} F. Bozzola, L. Brasco, The role of topology and capacity in some bounds for principal frequencies, J. Geom. Anal., {\bf 34} (2024), Paper No. 299, 46 pp.

\bibitem{BraBook} L. Brasco,  {\it Handbook of Calculus of Variations for absolute beginners.} Unitext, {\bf 163}, La Mat. per il 3+2, Springer, Cham, 2025.


\bibitem{BraBriPri} L. Brasco, L. Briani, F. Prinari, Extremals for Poincar\'e-Sobolev sharp constants in Steiner symmetric sets, preprint (2025), available at {\tt https://arxiv.org/abs/2505.11084}

\bibitem{BraFra} L. Brasco, G. Franzina, An overview on constrained critical points of Dirichlet integrals, Rend. Semin. Mat. Univ. Politec. Torino, {\bf 78} (2020), 7--50.

\bibitem{BraPriZag} L. Brasco, F. Prinari, A.C. Zagati, Sobolev embeddings and distance functions,  Adv. Calc. Var., {\bf 17} (2024), 1365--1398.
 
\bibitem{BPZ} L. Brasco, F. Prinari, A.C. Zagati, A comparison principle for the Lane-Emden equation and applications to geometric estimates, Nonlinear Anal. 220 (2022), Paper No. 112847, 41 pp.


\bibitem{Brezis} H. Brezis, {\it Functional Analysis, Sobolev Spaces and Partial Differential Equations.}  Universitext, Springer, New York, 2011.

\bibitem{BreLie} H. Br\'ezis, E. Lieb, A relation between pointwise convergence of functions and convergence of functionals,
	Proc. Amer. Math. Soc., {\bf 88} (1983), 486--490.



\bibitem{Cha} J. Chabrowski, Concentration-compactness principle at infinity and semilinear elliptic equations involving critical and subcritical Sobolev exponents, Calc. Var. Partial Differential Equations, {\bf 3} (1995), 493--512.
%

\bibitem{EP} G. Ercole, G. Pereira, The Asymptotics for the best Sobolev constants and their extremal functions, Math. Nachr., {\bf 289} (2016), 1433--1449.

\bibitem{Es} M. J. Esteban, Nonlinear elliptic problems in strip-like domains: symmetry of positive vortex rings, Nonlinear Anal., {\bf 7} (1983), 365--379.


\bibitem{FLL} J. Fr\"olich, E. H. Lieb, M. Loss, Stability of Coulomb systems with magnetic fields. I. The one-electron atom, Comm. Math. Phys., {\bf 104} (1986), 251--270. 



\bibitem{HL} R. Hynd, E. Lindgren, Extremal functions for Morrey's inequality in convex domains, Math. Ann., {\bf 375} (2019), 1721--1743.


\bibitem{KLP} B. Kawohl, M. Lucia, S. Prashanth, Simplicity of the principal eigenvalue for indefinite quasilinear problems, 
Adv. Differential Equations, {\bf 12} (2007), 407--434.

\bibitem{HuLi} D. Huang, Y. Li, A concentration-compactness principle at infinity and positive solutions of some quasilinear elliptic equations in unbounded domains, J. Math. Anal. Appl., {\bf 304} (2005), 58--73.



\bibitem{Li} E. H. Lieb, On the lowest eigenvalue of the Laplacian for the intersection of two domains,
Invent. Math., {\bf 74} (1983), 441--448.

\bibitem{LiHLS} E. H. Lieb, Sharp constants in the Hardy-Littlewood-Sobolev and related inequalities, Ann. of Math. (2), {\bf 118} (1983), 349--374.

\bibitem{LiLo} E. H. Lieb, M. Loss,  {\it Analysis}, Second Edition. American Mathematical Society, Providence, 2001.

\bibitem{LTW} W. C. Lien, S. Y. Tzeng, H. C. Wang, Existence of solutions of semilinear elliptic problems on unbounded domains, Differential Integral Equations, {\bf 6} (1993), 1281--1298.


\bibitem{Maz} V. Maz'ya, {\it Sobolev spaces with applications to elliptic partial differential equations}. Second, revised and augmented edition. Grundlehren der Mathematischen Wissenschaften [Fundamental Principles of Mathematical Sciences], {\bf 342}. Springer, Heidelberg, 2011. 

\bibitem{MaSh} V. Maz'ya,  M. Shubin, Can one see the fundamental frequency of a drum?, Lett. Math. Phys., {\bf 74},  (2005), 135--151.


\bibitem{Po1} G. Poliquin, Principal frequency of the $p-$Laplacian and the inradius of Euclidean domains, J. Topol. Anal., {\bf7} (2015), 505--511.

\bibitem{Sm} D. Smets, A concentration-compactness lemma with applications to singular eigenvalue problems, J. Funct. Anal., {\bf 167} (1999), 463--480.


\bibitem{St} M. Struwe, A global compactness result for elliptic boundary value problems involving limiting nonlinearities, Math. Z., {\bf 187} (1984), 511--517.

\bibitem{Vit} A. Vitolo, $H^{1,p}-$eigenvalues and $L^\infty-$estimates in quasicylindrical domains, Commun. Pure Appl. Anal., {\bf 10} (2011), 1315--1329.

\end{thebibliography}
\end{document}